\documentclass[reqno]{amsart}

\pdfoutput=1

\usepackage{amsmath,amsthm}
\usepackage{amssymb}
\usepackage{enumitem}
\usepackage{changepage}
\usepackage[dvipsnames]{xcolor}
\usepackage[colorlinks = true, citecolor = Green]{hyperref}

\theoremstyle{plain}
\newtheorem{thm}{Theorem}[section]

\newtheorem{lemma}[thm]{Lemma}

\theoremstyle{definition}
\newtheorem{defn}[thm]{Definition}
\newtheorem{alg}[thm]{Algorithm}

\newcommand{\1}{\mathbf{1}}
\newcommand{\e}{\varepsilon}
\newcommand\mubar{{\overline{\mu}}}
\newcommand{\D}{\mathcal{D}}
\newcommand{\R}{\mathbf{R}}
\newcommand{\X}{\mathcal{X}}
\newcommand{\Z}{\mathbf{Z}}
\newcommand{\TV}{\mathrm{TV}}

\DeclareMathOperator{\E}{\mathbf{E}}
\DeclareMathOperator{\Prob}{\mathbf{P}}

\title[Quantitative convergence via strong random times]{Quantitative convergence rates for reversible Markov chains via strong random times}

\author{Daniel C. Jerison}
\address{Department of Mathematical Sciences, Tel Aviv University, Tel Aviv 69978, Israel}
\email{jerison@math.tau.ac.il, dcjerison@gmail.com}

\date{August 18, 2019}

\begin{document}

\begin{abstract}
Let $(X_t)$ be a discrete time Markov chain on a general state space. It is well-known \cite{NT82,MT93,RR04} that if $(X_t)$ is aperiodic and satisfies a \emph{drift and minorization} condition, then it converges to its stationary distribution $\pi$ at an exponential rate. We consider the problem of computing upper bounds for the distance from stationarity in terms of the drift and minorization data.

Baxendale \cite{B05} showed that these bounds improve significantly if one assumes that $(X_t)$ is reversible with nonnegative eigenvalues (i.e.\ its transition kernel is a self-adjoint operator on $L^2(\pi)$ with spectrum contained in $[0,1]$). We identify this phenomenon as a special case of a general principle: for a reversible chain with nonnegative eigenvalues, any \emph{strong random time} gives direct control over the convergence rate. We formulate this principle precisely and deduce from it a stronger version of Baxendale's result.

Our approach is fully quantitative and allows us to convert drift and minorization data into explicit convergence bounds. We show that these bounds are tighter than those of \cite{R95,B05} when applied to a well-studied example.
\end{abstract}

\maketitle

\section{Introduction}
\label{sec:introduction}

\subsection{Overview}
\label{sec:overview}

This paper considers the problem of computing bounds for the exponential convergence of discrete time Markov chains on general state spaces. The problem arises in Markov chain Monte Carlo estimation (MCMC), in which one takes an approximate sample from a computationally intractable probability distribution $\pi$ by devising a Markov chain with $\pi$ as its stationary distribution and then running the chain until it has mixed.

Given a Markov transition kernel $P(x,dy)$ on the state space $\X$ with stationary distribution $\pi$, we aim to prove an explicit bound of the form
\begin{equation}
\label{eq:geom-ergodic}
\|P^t(x,\cdot) - \pi\|_\TV \leq F(x,t) \rho^t \qquad \text{for all $x \in \X$ and all integers $t \geq 0$,}
\end{equation}
where $F(x,t)$ is a polynomial in $t$ and $\rho < 1$. The notation $P^t(x,\cdot)$ is the law of the Markov chain started from $x$ after $t$ steps, and the total variation distance between two probability measures $\mu,\mu'$ on $\X$ is defined by
\[
\|\mu - \mu'\|_\TV = \sup_{S \subset \X} |\mu(S) - \mu'(S)|.
\]
(Here and throughout the paper, we restrict our attention to measurable subsets of $\X$.) A Markov chain satisfying \eqref{eq:geom-ergodic} is called \emph{geometrically ergodic}.

One of the most widely used techniques both to prove geometric ergodicity and to obtain formulas for $F,\rho$ is the method of \emph{drift and minorization}. As we discuss in Section \ref{sec:drift-minorization}, this method works by constructing a \emph{strong random time} for the Markov chain. Recall \cite{N84} that a randomized stopping time for a discrete time Markov chain $(X_t)_{t \geq 0}$ is a random time $T$ such that for each $n \geq 0$, the event $\{T = n\}$ is allowed to depend on the history $(X_0,\ldots,X_n)$ and additional randomness that plays no role in the trajectory of the chain after time $n$. In other words, given $(X_0,\ldots,X_n)$, the event $\{T = n\}$ and the future trajectory $(X_{n+1},X_{n+2},\ldots)$ must be conditionally independent.

\begin{defn}
\label{defn:strongrandom}
Let $(X_t)$ be a discrete time Markov chain on the state space $\X$, and fix a probability measure $\nu$ on $\X$. A \textbf{strong random time} for $(X_t)$ with measure $\nu$ is a randomized stopping time $T$ such that for every probability measure $\mu$ on $\X$,
\[
\Prob_\mu(X_n \in S \mid T = n) = \nu(S) \qquad \text{for all $n \geq 0$, $S \subset \X$.}
\]
(By $\Prob_\mu$ we mean the probability for the chain $(X_t)$ started from $X_0 \sim \mu$. We follow the convention that an equality of this form is trivially satisfied when the event being conditioned on has measure zero.)
\end{defn}

The main theoretical result in this paper, Theorem \ref{thm:strongrandom}, says that a strong random time $T$ for a reversible Markov chain with nonnegative eigenvalues directly controls its distance from stationarity. When the tail of the law of $T$ decays exponentially, the chain must converge at the same exponential rate (or faster). Many chains used in MCMC estimation are reversible, including Metropolis--Hastings chains and random scan Gibbs samplers \cite{RR04}. Often the eigenvalues are automatically nonnegative; if not, one can make them so by passing to a lazy version of the chain.

We use Theorem \ref{thm:strongrandom} to provide formulas for $F,\rho$ in \eqref{eq:geom-ergodic} for reversible Markov chains with nonnegative eigenvalues that satisfy a drift and minorization condition. These formulas are given in Theorem \ref{thm:TV-drift-minorize}. Theorem \ref{thm:V-norm} states the analogous bounds for $V$-norm convergence, which is stronger than convergence in total variation. For convenience, the statement of Theorem \ref{thm:V-norm} puts in one place all the steps to extract explicit convergence bounds in both total variation and $V$-norm from the drift and minorization data. In Section \ref{sec:example} we show that Theorems \ref{thm:TV-drift-minorize} and \ref{thm:V-norm} yield better numerical bounds than formulas of \cite{R95,B05} when applied to an example.

For the experienced reader, here is a brief comparison between our method and other approaches to the same problem. We assume that the Markov chain has a univariate rather than bivariate drift function, and we allow the small set to violate conditions like \cite[Proposition 11]{RR04} under which a bivariate drift function could easily be constructed. In this situation, Baxendale \cite{B05} has shown that the extra assumption of reversibility with nonnegative eigenvalues leads to substantially better convergence bounds. We derive this result as a corollary of Theorem \ref{thm:strongrandom}. While the proof in \cite{B05} requires a small set with $1$-step minorization, our probabilistic approach just as easily handles the general case of $m$-step minorization with $m > 1$. In addition, our bounds are somewhat sharper than those of \cite{B05} in the case $m = 1$ where both results apply. See the discussion in Section \ref{sec:related}.

The rest of Section \ref{sec:introduction} is organized as follows. Sections \ref{sec:intro-strongrandom}--\ref{sec:drift-minorization} state our main theorems on $L^2$ and total variation convergence. Section \ref{sec:related} describes related work. Sections \ref{sec:types-random}--\ref{sec:nonreversible} discuss the broader context surrounding our results. Finally, Section \ref{sec:outline} outlines the remainder of the paper.

\subsection*{Acknowledgements}

The results in this paper appeared first in my Ph.D. thesis \cite{J16}. Many thanks are due to my advisor, Persi Diaconis, for introducing me to the topic and shepherding my progress. I am supported by a Zuckerman Postdoctoral Scholarship.

\subsection{Convergence from strong random times}
\label{sec:intro-strongrandom}

We first introduce some standard notation. For background, see \cite{R84} (and \cite{RR97} regarding the space $L^2(\pi)$).

A Markov transition kernel $P(x,dy)$ on $\X$ acts on functions $f$ and measures $\mu$ respectively by
\[
Pf(x) = \int_\X f(y) P(x,dy), \qquad \mu P(S) = \int_\X P(x,S) \mu(dx).
\]
A \emph{stationary distribution} for $P$ is a probability measure $\pi$ such that $\pi P = \pi$. If $\mu$ is a measure on $\X$ and $f$ is a function on $\X$, we write
\[
\mu(f) = \int_\X f(x)\mu(dx).
\]

Let $\pi$ be a probability measure on $\X$ and consider the space $L^2(\pi)$ of real-valued functions on $\X$ having finite norm with respect to the inner product
\[
\langle f,g \rangle_\pi = \int_\X f(x)g(x)\pi(dx).
\]
The Markov chain $(X_t)$ is reversible with respect to $\pi$ if its transition kernel $P(x,dy)$ has the property that $\pi(dx)P(x,dy) = \pi(dy)P(y,dx)$ as measures on $\X \times \X$. This implies that $\pi P = \pi$. If we consider $P$ as an operator on $L^2(\pi)$ by $f \mapsto Pf$, then $(X_t)$ is reversible if and only if $P$ is self-adjoint. In that case, the $L^2(\pi)$ spectrum of $P$ is a subset of the interval $[-1,1]$. We say that $P$ (and $(X_t)$) have nonnegative eigenvalues if the spectrum is a subset of $[0,1]$.

The $L^2(\pi)$ distance between two probability measures $\mu,\mu'$ on $\X$ is defined as the $L^2(\pi)$ norm of the Radon--Nikodym derivative of $\mu - \mu'$ with respect to $\pi$, if $\mu - \mu'$ is absolutely continuous with respect to $\pi$, and $+\infty$ otherwise:
\[
\| \mu - \mu' \|_{L^2(\pi)}^2 = \int_\X \left[ \frac{d(\mu - \mu')}{d\pi}(x) \right]^2 \pi(dx).
\]
It is not hard to show that $\|\mu - \mu'\|_\TV \leq \frac{1}{2} \|\mu - \mu'\|_{L^2(\pi)}$.

\begin{thm}
\label{thm:strongrandom}
Let $(X_t)$ be a Markov chain on $\X$ with transition kernel $P$. Assume that $(X_t)$ is reversible with respect to the probability measure $\pi$ and that it has nonnegative eigenvalues. Suppose that $T$ is a strong random time for $(X_t)$ with measure $\nu$ such that $\Prob_\nu(T \geq 1) = 1$ and $\E_\nu[T] < \infty$. Also assume that $\Prob_\mu(T < \infty) = 1$ for all probability measures $\mu$ on $\X$. Then $\pi$ is the unique stationary distribution for $(X_t)$, and for all $t \geq 0$,
\[
\|P^t(\nu,\cdot) - \pi\|_{L^2(\pi)}^2 \leq \sum_{n = 2t+1}^\infty \Prob_\nu(T > n).
\]
\end{thm}

In this way, the distance from stationarity is controlled by the law of $T$. An exponential bound $\Prob_\nu(T > t) \leq A \rho^t$ implies that $\|P^t(\nu,\cdot) - \pi\|_{L^2(\pi)} \leq A' \rho^t$, where the leading constant changes but the exponential rate is the same.

Theorem \ref{thm:strongrandom} holds only when the Markov chain is started from the measure $\nu$ of the strong random time. Using an easy coupling argument, we can bound the total variation distance from stationarity for the chain started from any state $x \in \X$ in terms of $\Prob_\nu(T > t)$ and $\Prob_x(T > t)$. Below we give the result when these tail probabilities decay exponentially. For $r \in \R$ we use the notation $r_+ = r \vee 0$.

\begin{thm}
\label{thm:TV-strong}
Under the assumptions of Theorem \ref{thm:strongrandom}, fix $x \in \X$ and suppose that there are constants $A(\nu),A(x) < \infty$ and $\rho < 1$ such that $\Prob_\nu(T > t) \leq A(\nu) \rho^t$ and $\Prob_x(T > t) \leq A(x) \rho^t$ for all $t \geq 0$. Then
\[
\|P^t(x,\cdot) - \pi\|_\TV \leq F(x,t) \rho^t \qquad \text{for all } t \geq 0,
\]
where $F(x,t)$ is defined by setting
\begin{equation}
\label{eq:D-def}
D = \frac{1}{2} \sqrt{\frac{A(\nu) \rho}{1-\rho}}
\end{equation}
and then letting
\begin{equation}
\label{eq:F-def}
F(x,t) = \frac{1-\rho}{\rho} DA(x) t + (1-D)_+ A(x) + D.
\end{equation}
\end{thm}

We observe that $F(x,t)$ is linear in $t$ and that the exponential rate $\rho$ is the same as the decay rate of the law of $T$.

\subsection{Drift and minorization}
\label{sec:drift-minorization}

A Markov chain satisfies a drift and minorization condition if it has a \emph{drift function} with respect to a \emph{small set}.

\begin{defn}
\label{defn:drift}
Let $(X_t)$ be a Markov chain on $\X$ with transition kernel $P$. A \textbf{drift function} (or Foster--Lyapunov function) for $(X_t)$ with respect to a subset $C \subset \X$ is a function $V: \X \to [1,\infty)$ together with constants $\lambda < 1$ and $K < \infty$ such that
\[
PV(x) \leq \begin{cases} \lambda V(x), & x \notin C \\
 K, & x \in C.	
 \end{cases}
\]
\end{defn}

\begin{defn}
\label{defn:small}
Let $(X_t)$ be a Markov chain on $\X$ with transition kernel $P$. The subset $C \subset \X$ is a \textbf{small set} for $(X_t)$ if there are an integer $m \geq 1$, a constant $\e > 0$, and a probability measure $\nu$ on $\X$ such that
\[
P^m(x,S) \geq \e \nu(S) \qquad \text{for all $x \in C$, $S \subset \X$.}
\]
In this case, we say that $C$ has \emph{$m$-step minorization}.
\end{defn}

Let $C$ be a small set for $(X_t)$. For any initial measure $X_0 \sim \mu$, the following algorithm defines the joint law $\Prob_\mu$ of the chain $(X_t)$ and a strong random time $T$ with measure $\nu$.

\begin{alg}
\label{alg:T-construction}
Construction of the strong random time $T$.

\vspace{8pt}
\hrule
\vspace{8pt}

\begin{adjustwidth}{0.25in}{0.25in}
\begin{enumerate}[label=\arabic*.]
\item Start with $X_0 \sim \mu$ and $n = 0$.
\item While $X_n \notin C$:
	\begin{enumerate}[label=2\alph*.]
	\item Sample $X_{n+1} \sim P(X_n,\cdot)$ and replace $n$ with $n+1$.
	\end{enumerate}
\item[] End While
\item Flip a coin with heads probability $\e$ and tails probability $1-\e$.
\item If coin shows heads, sample $X_{n+m} \sim \nu$ and set $T = n+m$.
\item If coin shows tails, sample $X_{n+m}$ from the remainder distribution $\frac{1}{1-\e}[P^m(X_n, \cdot) - \e \nu(\cdot)]$.
\item If $m > 1$, sample $(X_{n+1},\ldots,X_{n+m-1})$ from the correct conditional distribution given $X_n$ and $X_{n+m}$.
\item Replace $n$ with $n+m$.
\item If coin showed tails, return to step $2$.
\item Sample the rest of the chain $(X_{n+1},X_{n+2},\ldots)$ from the correct conditional distribution given $X_n$.
\end{enumerate}
\noindent If the entire chain $(X_0,X_1,\ldots)$ is sampled without ever reaching step $4$, set $T = \infty$.
\end{adjustwidth}

\vspace{8pt}
\hrule
\vspace{6pt}
\end{alg}

In this algorithm, it is evident that the chain $(X_t)$ has the right law and that $T$ is a randomized stopping time for $(X_t)$. As well, the event $\{T = n\}$ is the same as the event that a coin was flipped at time $n-m$ and showed heads. The conditional law of $X_n$ given this event is $\nu$. Hence $T$ is a strong random time for $(X_t)$ with measure $\nu$.\footnote{To be fully rigorous, we could show that for any probability space $\Omega$ with a random variable $X = (X_0,X_1,\ldots)$ and a collection of measures $\{\Prob_\mu\}$ describing the law of $(X_t)$ started from $X_0 \sim \mu$, there is a probability space $\overline{\Omega}$ with a projection map $p: \overline{\Omega} \to \Omega$ and a random variable $T$ along with a collection of measures $\{\overline{\Prob}_\mu\}$, such that the push-forward of each $\overline{\Prob}_\mu$ by $p$ is $\Prob_\mu$ and the joint law of $X \circ p$ and $T$ under each $\overline{\Prob}_\mu$ is as described in Algorithm \ref{alg:T-construction}. This technical construction is carried out in \cite{J16}.}

So far, we have constructed $T$ using only that $C$ is a small set. A drift function for $(X_t)$ with respect to $C$ ensures that the chain will visit $C$ frequently, giving many chances for the coin to show heads and thereby bounding the law of $T$.
\begin{thm}
\label{thm:tail-bound}
Suppose that the Markov chain $(X_t)$ on $\X$ has a drift function with respect to a small set $C$. Let $V(x)$, $\lambda < 1$, $K < \infty$ be the data associated with the drift function, and let $\nu$, $m \geq 1$, $\e > 0$ be the data associated with the small set. Construct the strong random time $T$ with measure $\nu$ using Algorithm \ref{alg:T-construction}. Then $\Prob_\mu(T < \infty) = 1$ for all probability measures $\mu$ on $\X$. If $\mu(V) < \infty$, then
\[
\Prob_\mu(T > t) \leq \mu(V)^r \rho^{t+1-m} \qquad \text{for all } t \geq 0,
\]
where the formulas for $\rho,r$ are as follows. If $\e = 1$, then set $\rho = \lambda$ and $r = 1$. If $\e < 1$, then set
\begin{align}
\label{eq:B-def} B &= \frac{1 - \lambda^m}{1 - \lambda} (K - \lambda) + \lambda^m, \\
\label{eq:rho-def} \rho &= \lambda \vee \exp\left( \frac{-\log(1-\e) \log \lambda}{-m \log \lambda + \log(B - \e) - \log(1-\e)} \right), \\
\label{eq:r-def} r &= \frac{\log \rho}{\log \lambda}.
\end{align}
We have $\rho < 1$. In addition, the measure $\nu$ satisfies
\[
\nu(V) \leq \frac{B - (1-\e)}{\e}
\]
no matter whether $\e = 1$ or $\e < 1$.
\end{thm}

Theorem \ref{thm:tail-bound} shows that the tail probabilities of $T$ decay exponentially. This is exactly what is needed to apply Theorems \ref{thm:strongrandom} and \ref{thm:TV-strong}. We arrive at an explicit form of \eqref{eq:geom-ergodic} when the chain $(X_t)$ is reversible with nonnegative eigenvalues and satisfies a drift and minorization condition.

\begin{thm}
\label{thm:TV-drift-minorize}
Let $(X_t)$ be a Markov chain on $\X$ with transition kernel $P$. Assume that $(X_t)$ is reversible with respect to the probability measure $\pi$ and that it has nonnegative eigenvalues. Also assume that $(X_t)$ has a drift function with respect to a small set $C \subset \X$. Let $V(x)$, $\lambda < 1$, $K < \infty$ be the data associated with the drift function, and let $\nu$, $m \geq 1$, $\e > 0$ be the data associated with the small set. Then
\[
\|P^t(x,\cdot) - \pi\|_\TV \leq F(x,t) \rho^t \qquad \text{for all $x \in \X$, $t \geq 0$,}
\]
where $\rho < 1$ and $F(x,t)$ (which is linear in $t$) are defined by the following recipe.
\begin{enumerate}[label=\arabic*.]
\item Define $B$ as in \eqref{eq:B-def}.
\item If $\e = 1$, then set $\rho = \lambda$ and $r = 1$.
\item If $\e < 1$, then define $\rho$ as in \eqref{eq:rho-def} and $r$ as in \eqref{eq:r-def}.
\item Define the constant $A(\nu)$ and the function $A(x)$ by
\[
A(\nu) = \left[ \frac{B - (1-\e)}{\e} \right]^r \rho^{1-m}, \qquad A(x) = V(x)^r \rho^{1-m}.
\]
\item Define $D$ as in \eqref{eq:D-def}.
\item Define $F(x,t)$ as in \eqref{eq:F-def}.
\end{enumerate}
\end{thm}

\begin{proof}
Apply Theorem \ref{thm:tail-bound} followed by Theorem \ref{thm:TV-strong}.
\end{proof}

Theorem \ref{thm:V-norm} strengthens Theorem \ref{thm:TV-drift-minorize} to a bound of the form
\begin{equation}
\label{eq:G-def}
\|P^t(x,\cdot) - \pi\|_V \leq G(x,t) \rho^t \qquad \text{for all $x \in \X$, $t \geq 0$,}
\end{equation}
where $V$ is the drift function from Definition \ref{defn:drift} and the value of $\rho$ is the same as in Theorem \ref{thm:TV-drift-minorize}. The \emph{$V$-norm distance} between two probability measures $\mu,\mu'$ on $\X$ is defined to be
\[
\|\mu - \mu'\|_V = \sup_{|f| \leq V} |\mu(f) - \mu'(f)|.
\]
Note that since $V \geq 1$, we have
\[
2 \|\mu - \mu'\|_\TV = \sup_{|f| \leq 1} |\mu(f) - \mu'(f)| \leq \|\mu - \mu'\|_V.
\]

\subsection{Related work}
\label{sec:related}

In the case $m = 1$, the decay rate $\rho$ of the law of $T$ was identified by Roberts and Tweedie \cite{RT99} (who use the notation $\beta_{\textrm{RT}} = \rho^{-1}$). Theorem 4.1(i) of \cite{RT99} is equivalent to a bound of the form
\[
\Prob_\mu(T > t) \leq (\textrm{const}) \mu(V)^r t \rho^t.
\]
Theorem \ref{thm:tail-bound} slightly improves this result by removing the factor of $t$ and generalizing to the case $m > 1$.

The most important feature of Theorem \ref{thm:TV-drift-minorize} and its $V$-norm version, Theorem \ref{thm:V-norm}, is that the exponential rate $\rho$ is the same as the decay rate in Theorem \ref{thm:tail-bound}. As we will see in Section \ref{sec:nonreversible}, this conclusion can only be drawn for reversible Markov chains with nonnegative eigenvalues and does not hold in general. Baxendale \cite{B05} was the first to observe this consequence of reversibility.\footnote{For a key argument, he credits a comment of Meyn on a previous draft of \cite{B05}.} In the case $m = 1$, Theorem \ref{thm:V-norm} is very similar to \cite[Theorem 1.3]{B05}: both theorems have the same hypotheses, and both prove $V$-norm convergence of the chain with the same exponential rate $\rho$.

We will see in Section \ref{sec:example} that Theorem \ref{thm:V-norm} yields better numerical bounds than \cite[Theorem 1.3]{B05}. Indeed, in the regime $\rho > \lambda$, Theorem \ref{thm:V-norm} proves a bound \eqref{eq:G-def} where $G(x,t)$ is linear in $t$, while \cite[Theorem 1.3]{B05} is equivalent to \eqref{eq:G-def} with $G(x,t)$ cubic in $t$.

The method of proof in \cite{B05} uses analytic properties of generating functions for renewal sequences. In principle the argument could be extended to the case $m > 1$, but the resulting bound on the exponential convergence rate of the chain would be worse than the rate $\rho$ from Theorem \ref{thm:tail-bound}. Intuitively, this is because the law of $T$ might introduce artificial periodicity. Our approach using Theorem \ref{thm:strongrandom} is probabilistic and puts all cases $m \geq 1$ on the same footing.

\subsection{Types of random times}
\label{sec:types-random}

Strong random times first appeared in the pioneering work of Athreya and Ney \cite{AN78}. They carried out Algorithm \ref{alg:T-construction} in the case $m = 1$ for a Harris recurrent and strongly aperiodic Markov chain $(X_t)$ with stationary distribution $\pi$. (We define strong aperiodicity in Section \ref{sec:nonreversible}.) Using the strong random time $T$, they applied standard techniques of discrete renewal theory to show that $(X_t)$ converges to $\pi$ in total variation. Independently, Nummelin \cite{N78,N84} found an equivalent formulation of the same argument.

The term ``strong random time'' was coined by Miclo \cite{M10} in analogy with \emph{strong stationary times}. A strong stationary time for the chain $(X_t)$ is a strong random time whose measure $\nu$ is the stationary distribution of $(X_t)$. Strong stationary times were introduced by Aldous and Diaconis \cite{AD86,AD87}, who noted the connection with \cite{AN78,N78,N84}. They are now a well-established approach to bound the mixing time of finite Markov chains (see \cite[Ch.~6]{LP17} and \cite{P97}).

When $m = 1$, but not in general when $m > 1$, the strong random time $T$ from Algorithm \ref{alg:T-construction} is in fact a \emph{regeneration time} for $(X_t)$.

\begin{defn}
\label{defn:regeneration}
Let $(X_t)$ be a Markov chain on $\X$. A \textbf{regeneration time} for $(X_t)$ with measure $\nu$ is a randomized stopping time $T$ such that for every initial measure $\mu$ on $\X$,
\[
\Prob_\mu(X_n \in S \mid T = n, X_0,\ldots,X_{n-1}) = \nu(S) \qquad \text{for all $n \geq 0$, $S \subset \X$.}
\]
\end{defn}

Conditioned on $T = n$, the future trajectory $(X_n,X_{n+1},\ldots)$ has the law of the chain started from $\nu$ and is conditionally independent of the history $(X_0,\ldots,X_{n-1})$. For this reason we say that the chain regenerates at time $T$.

Given a regeneration time $T$, one can split the sample path $(X_0,X_1,\ldots)$ into an initial part $(X_0,\ldots,X_{T-1})$ followed by a sequence of i.i.d.\ tours between successive regenerations. This is useful for proving limit theorems about ergodic averages \cite{RR04,BLL08}. As pointed out by \cite{BLL08}, the full strength of Definition \ref{defn:regeneration} is required for the tours to be independent. When $T$ is a strong random time that is not a regeneration time, the sequence of tours is $1$-dependent. (See \cite[Theorem 17.3.1]{MT93} for a proof when $T$ is defined by Algorithm \ref{alg:T-construction} with $m > 1$, and \cite[Proposition 3.7]{J16} for the general case).

\subsection{Non-reversible chains}
\label{sec:nonreversible}

In this subsection we discuss variants of Theorem \ref{thm:TV-drift-minorize} that do not require the Markov chain $(X_t)$ to be reversible. For more background on the convergence theory of general state space Markov chains, we refer the reader to the detailed development in \cite{MT93}.

Nummelin and Tuominen \cite{NT82} showed that an aperiodic Markov chain satisfying a drift and minorization condition must be geometrically ergodic. Meyn and Tweedie \cite{MT94} obtained the first quantitative version of this result, with explicit formulas for $G,\rho$ in \eqref{eq:G-def}. They assume that the Markov chain is \emph{strongly aperiodic}: the small set $C$ has $1$-step minorization (that is, $m = 1$) and the measure $\nu$ satisfies $\nu(C) \geq \beta / \e$ for some constant $\beta > 0$. This assumption immediately implies that $\Prob_\nu(T = 1) \geq \beta$. Subsequent work of Baxendale \cite{B05} improved the bound of \cite{MT94}. (In Section \ref{sec:related} we discussed \cite[Theorem 1.3]{B05} for reversible chains; here we consider \cite[Theorem 1.1]{B05}, which does not require reversibility.) Most recently, Bednorz \cite{B13} has sharpened the proof of \cite[Theorem 1.1]{B05}, yielding convergence bounds which are tighter but difficult to compute unless the value of $\pi(C)$ is known exactly.

For the example considered in Section \ref{sec:example}, we will see that the numerical bound provided by \cite[Theorem 1.1]{B05} is extremely conservative. Here we explain why this must be the case. Our explanation is adapted from the similar discussion in \cite[Section 3.1]{B05}. Suppose that the Markov chain $(X_t)$ with stationary distribution $\pi$ satisfies a drift and minorization condition and is strongly aperiodic. What can we conclude about the convergence rate solely from the drift and minorization data and the aperiodicity parameter $\beta$? Below, we give an example where the minorization is well-behaved ($m = 1$, $\e = 1$, $\beta = \frac{1}{2}$) and the drift function is bounded ($\sup_{x \in \X} V(x) \leq 3$) with a drift rate $\lambda$ that can be chosen arbitrarily close to $1$. Theorem \ref{thm:tail-bound} implies that $\Prob_\mu(T > t) \leq 3 \lambda^t$ for every initial measure $\mu$. On the other hand, the Markov chain converges much more slowly: if $\rho_\TV$ is the optimal rate in \eqref{eq:geom-ergodic}, then $1 - \rho_\TV$ is proportional to $(1 - \lambda)^3$ as $\lambda \nearrow 1$. Thus, any analogue to Theorem \ref{thm:TV-drift-minorize} in which the assumption of reversibility is removed, as in \cite{MT94,B05,B13}, must have a significantly worse upper bound on the rate of convergence.

The example is the nearly periodic chain on $\X = \Z/N\Z$ with transition matrix
\[
P(j,k) = \begin{cases} 1, & j \neq 0 \text{ and } k = j-1 \\
 \frac{1}{2}, & j = 0 \text{ and } k \in \{0,N-1\}	 \\
 0, & \text{otherwise.}
 \end{cases}
\]
We set $C = \{0\}$ and use the drift function $V(j) = (1 - \frac{1}{N})^{-j}$, which satisfies Definition \ref{defn:drift} with $\lambda = 1 - \frac{1}{N}$ and $K = (1+e)/2$. For minorization we take $m = 1$, $\e = 1$, and $\nu(0) = \nu(N-1) = \frac{1}{2}$, so that Definition \ref{defn:small} is satisfied and the chain is strongly aperiodic with $\beta = \frac{1}{2}$.

The following heuristic argument shows that the Markov chain should take order $N^3$ steps to mix. Imagine a random walker moving around the circle $\Z/N\Z$ according to $P$. Every time it reaches zero, it pauses for a random number of time steps before continuing around. The amount of time that the walker pauses at zero is a geometric random variable with parameter $\frac{1}{2}$. In order for $\|P^t(x,\cdot)-\pi\|_\TV$ to be small, the total amount of time paused at zero (which is essentially a sum of independent $\mathrm{Geometric}(\frac{1}{2})$ random variables) must have standard deviation of at least order $N$. This will not happen until the random walker has taken order $N^2$ trips around the circle, so $t$ must be of order $N^3$. A computation in \cite[Section 3.1]{B05} confirms this argument by verifying that $1 - \rho_\TV$ is proportional to $1/N^3$.

An alternative to the approach of \cite{MT94,B05,B13} is the \emph{bivariate drift} method developed by Rosenthal \cite{R95}. See \cite{R02,RR04} for an exposition of this technique and \cite{F02} for a more flexible and powerful version. To use the method, one finds a small set $C$ and a so-called bivariate drift function with respect to $C$. Several papers \cite{R95,R95b,JH01,JH04,MH04} have followed this procedure to prove useful convergence bounds for Markov chains of practical significance in MCMC. In Section \ref{sec:example} we compare the bivariate drift method against Theorem \ref{thm:TV-drift-minorize} using an example treated in \cite{R95}, which is reversible with nonnegative eigenvalues. We find that Theorem \ref{thm:TV-drift-minorize} gives a tighter convergence bound. For more details on the relationship between univariate drift functions (as in Definition \ref{defn:drift}) and bivariate drift functions, see the discussion in \cite[Ch.~2]{J16}.

We finish by briefly mentioning Markov chains whose convergence rate is polynomial rather than exponential. In this setting, the assumption of reversibility with nonnegative eigenvalues does \emph{not} seem to improve the convergence bounds. Rather, Theorem \ref{thm:strongrandom} is outperformed by \cite[Theorem 3.4]{JR02}, which does not require reversibility. See \cite{FM03,DFMS04,DMS07} for more about chains with subexponential convergence rates.

\subsection{Outline}
\label{sec:outline}

Section \ref{sec:proofs-strongrandom} proves Theorems \ref{thm:strongrandom}--\ref{thm:TV-strong}, and Section \ref{sec:tail} proves Theorem \ref{thm:tail-bound}. As we have seen, Theorem \ref{thm:TV-drift-minorize} follows immediately from combining Theorems \ref{thm:TV-strong} and \ref{thm:tail-bound}. Section \ref{sec:V-norm} then strengthens Theorem \ref{thm:TV-drift-minorize} to a bound on the $V$-norm distance from stationarity. Theorem \ref{thm:V-norm} collects in one place the formulas for both the total variation bound (Theorem \ref{thm:TV-drift-minorize}) and the new $V$-norm bound. Section \ref{sec:example} applies Theorems \ref{thm:TV-drift-minorize} and \ref{thm:V-norm} to a Markov chain considered by \cite{R95} and compares the resulting numerical bounds against those of \cite{R95,B05}.

\section{Proofs for strong random times}
\label{sec:proofs-strongrandom}

In this section we prove Theorems \ref{thm:strongrandom}--\ref{thm:TV-strong}, which underlie the bounds in Theorems \ref{thm:TV-drift-minorize} and \ref{thm:V-norm}. Lemma \ref{lemma:core} below contains the core of the argument.

We begin by identifying the stationary distribution of the Markov chain.

\begin{lemma}
\label{lemma:stationary}
Let $(X_t)$ be a Markov chain on $\X$ with transition kernel $P$. Suppose that $T$ is a strong random time for $(X_t)$ with measure $\nu$ satisfying the conditions of Theorem \ref{thm:strongrandom}: $\Prob_\nu(T \geq 1) = 1$, $\E_\nu[T] < \infty$, and $\Prob_\mu(T < \infty) = 1$ for all probability measures $\mu$ on $\X$. Then $(X_t)$ has unique stationary distribution $\pi$ given by
\[
\pi(S) = \frac{1}{\E_\nu[T]} \sum_{n=0}^\infty \Prob_\nu(X_n \in S,\, T > n).
\]	
\end{lemma}

In fact, the proof below goes through even if we replace the condition $\Prob_\mu(X_n \in S \mid T = n) = \nu(S)$ with the weaker condition $\Prob_\mu(X_T \in S) = \nu(S)$.

\begin{proof}
First we show that the given $\pi$ is stationary. We observe that
\[
\pi(\X) = \frac{1}{\E_\nu[T]} \sum_{n=0}^\infty \Prob_\nu(T > n) = 1.
\]
Since $T$ is a randomized stopping time for $(X_t)$, the event $\{T > n\}$ is conditionally independent of $X_{n+1}$ given $(X_0,\ldots,X_n)$. Hence for any $S \subset \X$ we have
\[
\Prob_\nu(X_{n+1} \in S \mid X_n = x,\, T > n) = \Prob_\nu(X_{n+1} \in S \mid X_n = x) = P(x,S)
\]
and therefore
\[
\Prob_\nu(X_{n+1} \in S \mid T > n) = \int_\X P(x,S) \Prob_\nu(X_n \in dx \mid T > n).
\]
It follows that
\[
\begin{split}
\pi P(S) &= \frac{1}{\E_\nu[T]} \sum_{n=0}^\infty \Prob_\nu(T > n) \int_\X P(x,S) \Prob_\nu(X_n \in dx \mid T > n) \\
&= \frac{1}{\E_\nu[T]} \sum_{n=0}^\infty \Prob_\nu(X_{n+1}\in S,\, T > n) \\
&= \frac{1}{\E_\nu[T]} \sum_{n=0}^\infty \Big[ \Prob_\nu(X_{n+1}\in S,\, T = n+1) + \Prob_\nu(X_{n+1}\in S,\, T > n+1) \Big] \\
&= \frac{1}{\E_\nu[T]} \Prob_\nu(X_T\in S) + \frac{1}{\E_\nu[T]} \sum_{n=1}^\infty \Prob_\nu(X_n\in S,\, T > n).
\end{split}
\]
Because $T$ is a strong random time with measure $\nu$, we know that
\[
\Prob_\nu(X_T \in S) = \nu(S) = \Prob_\nu(X_0 \in S) = \Prob_\nu(X_0 \in S,\, T > 0).
\]
Thus,
\[
\pi P(S) = \frac{1}{\E_\nu[T]} \sum_{n=0}^\infty \Prob_\nu(X_n\in S,\, T > n) = \pi(S).
\]

To prove uniqueness, suppose for contradiction that $\pi_1$ and $\pi_2$ are two different stationary distributions for $(X_t)$. Then, using the Hahn decomposition theorem \cite{C13}, we can partition $\X$ into disjoint subsets $\X = \X_+ \sqcup \X_-$ such that $\pi_1 - \pi_2$ is a positive measure on $\X_+$ and a negative measure on $\X_-$. Since $\pi_1$ and $\pi_2$ are different, $(\pi_1-\pi_2)(\X_+) = (\pi_2-\pi_1)(\X_-) > 0$. Define the probability measures $\mu_1,\mu_2$ on $\X$ by
\[
\mu_1(S) = \frac{(\pi_1-\pi_2)(S \cap \X_+)}{(\pi_1-\pi_2)(\X_+)}, \qquad \mu_2(S) = \frac{(\pi_2-\pi_1)(S \cap \X_-)}{(\pi_2-\pi_1)(\X_-)}.
\]
Then $\mu_1$ and $\mu_2$ are also stationary distributions for $(X_t)$, by the following argument. Since $\pi_1-\pi_2$ is an invariant measure for $(X_t)$,
\[
\begin{split}
(\pi_1-\pi_2)(\X_+) &= \int_{\X_+} P(x,\X_+)(\pi_1-\pi_2)(dx) - \int_{\X_-} P(x,\X_+)(\pi_2-\pi_1)(dx) \\
&\leq \int_{\X_+} (\pi_1-\pi_2)(dx) - 0 = (\pi_1-\pi_2)(\X_+).
\end{split}
\]
Hence the inequality in the middle is actually equality, and we have
\begin{align}
\int_{\X_+} P(x,\X_+)(\pi_1-\pi_2)(dx) &= \int_{\X_+} (\pi_1-\pi_2)(dx), \label{eq:NA1} \\
\int_{\X_-} P(x,\X_+)(\pi_2-\pi_1)(dx) &= 0. \label{eq:NA2}
\end{align}
From \eqref{eq:NA1} it follows that
\begin{equation}
\label{eq:NA3}
0 = \int_{\X_+} [1 - P(x,\X_+)](\pi_1-\pi_2)(dx) = \int_{\X_+} P(x,\X_-)(\pi_1-\pi_2)(dx).	
\end{equation}

Given $S \subset \X$, write $S_+ = S \cap \X_+$. Using that $\pi_1 - \pi_2$ is an invariant measure for $(X_t)$, followed by \eqref{eq:NA2}, we compute
\[
(\pi_1-\pi_2)(S_+) = \int_\X P(x,S_+)(\pi_1-\pi_2)(dx) = \int_{\X_+} P(x,S_+)(\pi_1-\pi_2)(dx).
\]
We also have by \eqref{eq:NA3} that
\[
\int_{\X_+} P(x,S)(\pi_1-\pi_2)(dx) = \int_{\X_+} P(x,S_+)(\pi_1-\pi_2)(dx).
\]
Therefore,
\[
\int_{\X_+} P(x,S)(\pi_1-\pi_2)(dx) = (\pi_1-\pi_2)(S_+)
\]
and we conclude that
\[
\mu_1 P(S) = \frac{1}{(\pi_1-\pi_2)(\X_+)} \int_{\X_+} P(x,S)(\pi_1-\pi_2)(dx) = \frac{(\pi_1-\pi_2)(S_+)}{(\pi_1-\pi_2)(\X_+)} = \mu_1(S).
\]
This proves that $\mu_1$ is stationary, and the argument for $\mu_2$ is the same.

Since $T$ is almost surely finite started from $\mu_1$,
\[
\sum_{n=1}^\infty \Prob_{\mu_1}(X_n\in \X_-,\, T = n) = \Prob_{\mu_1}(X_T\in \X_-) = \nu(\X_-).
\]
However,
\[
\sum_{n=1}^\infty \Prob_{\mu_1}(X_n\in \X_-,\, T = n) \leq \sum_{n=1}^\infty \Prob_{\mu_1}(X_n\in \X_-) = \sum_{n=1}^\infty \mu_1(\X_-) = 0.
\]
Thus $\nu(\X_-) = 0$, and by parallel reasoning $\nu(\X_+) = 0$ as well. This is impossible since $\nu(\X) = 1$. Therefore, the stationary distribution $\pi$ is unique.
\end{proof}

We will prove Theorem \ref{thm:strongrandom} by finding a function $f \geq 0$ on $\X$ such that the sequence $\E_\nu[f(X_t)]$ controls the convergence of $P^t(\nu,\cdot)$ to $\pi$ in $L^2(\pi)$ distance. Using that $(X_t)$ is reversible with nonnegative eigenvalues, we will show that:
\begin{enumerate}
\item $\|P^t(\nu,\cdot) - \pi\|_{L^2(\pi)}^2 = \E_\nu[f(X_{2t})] - 1$ for all $t \geq 0$;
\item The sequence $\E_\nu[f(X_t)]$ is nonincreasing and converges to $1$.
\end{enumerate}
Given these properties, Theorem \ref{thm:strongrandom} is an immediate consequence of the following lemma, which does not require reversibility and is proved via summation by parts.

\begin{lemma}
\label{lemma:core}
Let $T$ be a strong random time with measure $\nu$ for the Markov chain $(X_t)$ on $\X$. Assume that $\Prob_\nu(T \geq 1) = 1$ and $\E_\nu[T] < \infty$. Let $f \geq 0$ be a function on $\X$ such that each $\E_\nu[f(X_t)] < \infty$ and the sequence $\E_\nu[f(X_t)]$ is nonincreasing in $t$. Denote the limit of the sequence by $\E_\nu[f(X_\infty)]$. Then for all $t \geq 0$,
\[
\E_\nu[f(X_t)] - \E_\nu[f(X_\infty)] \leq \E_\nu[f(X_\infty)] \sum_{n=t+1}^\infty \Prob_\nu(T>n).
\]
\end{lemma}

\begin{proof}
Fix a positive integer $n$. Since $T$ is a strong random time with measure $\nu$,
\[
\E_\nu[f(X_n),\, T\leq n] = \sum_{j=0}^{n-1} \E_\nu[f(X_j)]\Prob_\nu(T = n-j).
\]
Apply summation by parts to obtain
\begin{multline*}
\E_\nu[f(X_n),\, T\leq n] = \E_\nu[f(X_{n-1})] - \E_\nu[f(X_0)] \Prob_\nu(T > n) \\
+ \sum_{j=1}^{n-1} \Big( \E_\nu[f(X_{j-1})] - \E_\nu[f(X_j)] \Big) \Prob_\nu(T>n-j).
\end{multline*}
Rearranging this equation, we have
\begin{multline*}
\sum_{j=1}^{n} \Big( \E_\nu[f(X_{j-1})] - \E_\nu[f(X_j)] \Big) \Prob_\nu(T>n-j) \\
= \E_\nu[f(X_0)] \Prob_\nu(T>n) - \E_\nu[f(X_n),\, T>n]
\end{multline*}
and therefore
\begin{equation}
\label{eq:renewal-identity}
\sum_{j=1}^{n} \Big( \E_\nu[f(X_{j-1})] - \E_\nu[f(X_j)] \Big) \Prob_\nu(T>n-j) \leq \E_\nu[f(X_0)] \Prob_\nu(T>n).
\end{equation}

Each term $\E_\nu[f(X_{j-1})] - \E_\nu[f(X_j)]$ is nonnegative. Summing \eqref{eq:renewal-identity} from $n=1$ to $\infty$ gives
\[
\Big( \E_\nu[f(X_0)] - \E_\nu[f(X_\infty)] \Big) \E_\nu[T] \leq \E_\nu[f(X_0)](\E_\nu[T] - 1),
\]
which means that $\E_\nu[f(X_0)] \leq \E_\nu[f(X_\infty)]\E_\nu[T]$.

Fix $t\geq 0$. Summing \eqref{eq:renewal-identity} from $n=t+1$ to $\infty$, the left side is
\[
\sum_{j=1}^\infty \Big( \E_\nu[f(X_{j-1})] - \E_\nu[f(X_j)] \Big) \sum_{n = (t+1) \vee j}^\infty \Prob_\nu(T>n-j),
\]
which is greater than or equal to
\begin{multline*}
\sum_{j=t+1}^\infty \Big( \E_\nu[f(X_{j-1})] - \E_\nu[f(X_j)] \Big) \sum_{n=j}^\infty \Prob_\nu(T>n-j) \\
= \Big( \E_\nu[f(X_t)] - \E_\nu[f(X_\infty)] \Big) \E_\nu[T].
\end{multline*}
The right side of the sum of \eqref{eq:renewal-identity} from $n=t+1$ to $\infty$ is
\[
\E_\nu[f(X_0)] \sum_{n=t+1}^\infty \Prob_\nu(T>n) \leq \E_\nu[f(X_\infty)]\E_\nu[T] \sum_{n=t+1}^\infty \Prob_\nu(T>n).
\]
Hence,
\[
\E_\nu[f(X_t)] - \E_\nu[f(X_\infty)] \leq \E_\nu[f(X_\infty)] \sum_{n=t+1}^\infty \Prob_\nu(T>n). \qedhere
\]
\end{proof}

\begin{proof}[Proof of Theorem \ref{thm:strongrandom}]
Since the operator $P$ on $L^2(\pi)$ is self-adjoint and its spectrum is a subset of $[0,1]$, we have for each $f \in L^2(\pi)$ that the sequence $\langle P^t f, f \rangle_\pi$ is nonnegative and nonincreasing. This can be seen by writing
\[
\langle P^t f, f \rangle_\pi = \int_{[0,1]} \lambda^t \psi_f(d\lambda)
\]
where $\psi_f$ is the spectral measure associated with $f$ \cite[Section VII.2]{RS72}. Alternatively, it can be shown by a series of substitutions for $g$ in the inequalities $\langle Pg,g \rangle_\pi \geq 0$ and $\langle Pg,Pg \rangle_\pi \leq \langle g,g \rangle_\pi$.

By Lemma \ref{lemma:stationary}, for $S \subset \X$,
\[
\pi(S) = \frac{1}{\E_\nu[T]} \sum_{n=0}^\infty \Prob_\nu(X_n\in S,\, T > n) \geq \frac{1}{\E_\nu[T]} \Prob_\nu(X_0\in S,\, T > 0) = \frac{\nu(S)}{\E_\nu[T]}.
\]
Thus the Radon--Nikodym derivative of $\nu$ with respect to $\pi$ satisfies $d\nu / d\pi \leq \E_\nu[T]$ $\pi$-almost everywhere. In particular, $d\nu / d\pi \in L^2(\pi)$.

Fix $t\geq 0$. For any $f\in L^2(\pi)$,
\[
\int_\X f(x)P^t(\nu,dx) = \int_\X (P^t f)(x) \frac{d\nu}{d\pi}(x)\pi(dx) = \int_\X f(x) \left( P^t \frac{d\nu}{d\pi} \right)(x)\pi(dx),
\]
where the second equality used reversibility of $P$. This means precisely that $P^t(\nu,\cdot)$ is absolutely continuous with respect to $\pi$, with Radon--Nikodym derivative
\begin{equation}
\label{eq:Radon--Nikodym}
\frac{dP^t(\nu,\cdot)}{d\pi}(x) = \left( P^t \frac{d\nu}{d\pi} \right)(x).
\end{equation}
We have
\begin{equation}
\label{eq:crossterm}
\bigg\langle P^t \frac{d\nu}{d\pi},\, 1 \bigg\rangle_\pi = \bigg\langle \frac{d\nu}{d\pi},\, P^t 1 \bigg\rangle_\pi = \bigg\langle \frac{d\nu}{d\pi},\, 1 \bigg\rangle_\pi = 1.
\end{equation}
Using \eqref{eq:Radon--Nikodym} in the definition of $L^2(\pi)$ distance, and then \eqref{eq:crossterm},
\[
\begin{split}
\|P^t(\nu,\cdot) - \pi\|_{L^2(\pi)}^2 &= \bigg\langle P^t \frac{d\nu}{d\pi} - 1,\, P^t \frac{d\nu}{d\pi} - 1 \bigg\rangle_\pi = \bigg\langle P^t \frac{d\nu}{d\pi},\, P^t \frac{d\nu}{d\pi} \bigg\rangle_\pi - 1 \\
&= \bigg\langle P^{2t} \frac{d\nu}{d\pi},\, \frac{d\nu}{d\pi} \bigg\rangle_\pi - 1.
\end{split}
\]

Consider the sequence
\[
\bigg\langle P^t \frac{d\nu}{d\pi},\, \frac{d\nu}{d\pi} \bigg\rangle_\pi = \int_\X \left( P^t \frac{d\nu}{d\pi} \right)(x) \nu(dx) = \E_\nu \left[ \frac{d\nu}{d\pi}(X_t) \right],
\]
which is nonincreasing by the discussion at the start of the proof. Let
\[
a = \lim_{t\to\infty} \E_\nu \left[ \frac{d\nu}{d\pi}(X_t) \right].
\]
The following argument shows that $a=1$. For any $t\geq 0$,
\[
\E_\pi \left[ \frac{d\nu}{d\pi}(X_t) \right] = \E_\pi \left[ \frac{d\nu}{d\pi}(X_0) \right] = \int_\X \frac{d\nu}{d\pi}(x) \pi(dx) = \int_\X \nu(dx) = 1.
\]
As well,
\begin{equation}
\label{eq:to-take-limit}
\E_\pi \left[ \frac{d\nu}{d\pi}(X_t) \right] = \E_\pi \left[ \frac{d\nu}{d\pi}(X_t),\, T>t \right] + \sum_{s=0}^\infty \Prob_\pi(T=s) \E_\nu \left[ \frac{d\nu}{d\pi}(X_{t-s}) \right]
\end{equation}
where $\E_\nu \left[ \frac{d\nu}{d\pi}(X_{t-s}) \right]$ is taken to be zero when $t-s < 0$. Take the limit as $t\to\infty$ of \eqref{eq:to-take-limit}. The left side is $1$. For the first part of the right side,
\[
\lim_{t\to\infty} \E_\pi \left[ \frac{d\nu}{d\pi}(X_t),\, T>t \right] \leq \lim_{t\to\infty} \E_\nu[T] \Prob_\pi(T>t) = 0.
\]
For the second part of the right side, use dominated convergence to interchange the sum and the limit. This is legal because for all $t$,
\[
\Prob_\pi(T=s) \E_\nu \left[ \frac{d\nu}{d\pi}(X_{t-s}) \right] \leq \E_\nu[T] \Prob_\pi(T=s),
\]
and
\[
\sum_{s=0}^\infty \E_\nu[T] \Prob_\pi(T=s) = \E_\nu[T] < \infty.
\]
Hence
\[
\begin{split}
\lim_{t\to\infty} \sum_{s=0}^\infty \Prob_\pi(T=s) \E_\nu \left[ \frac{d\nu}{d\pi}(X_{t-s}) \right] &= \sum_{s=0}^\infty \lim_{t\to\infty} \Prob_\pi(T=s) \E_\nu \left[ \frac{d\nu}{d\pi}(X_{t-s}) \right] \\
&= \sum_{s=0}^\infty \Prob_\pi(T=s) \cdot a = a.
\end{split}
\]
So taking the limit as $t\to\infty$ of \eqref{eq:to-take-limit} yields $1 = 0 + a$.

Lemma \ref{lemma:core} with $f = d\nu / d\pi$ gives
\[
\bigg\langle P^t \frac{d\nu}{d\pi},\, \frac{d\nu}{d\pi} \bigg\rangle_\pi - 1 \leq \sum_{n=t+1}^\infty \Prob_\nu(T>n).
\]
It follows that
\[
\|P^t(\nu,\cdot) - \pi\|_{L^2(\pi)}^2 = \bigg\langle P^{2t} \frac{d\nu}{d\pi},\, \frac{d\nu}{d\pi} \bigg\rangle_\pi - 1 \leq \sum_{n=2t+1}^\infty \Prob_\nu(T>n). \qedhere
\]
\end{proof}

\begin{proof}[Proof of Theorem \ref{thm:TV-strong}]
We first compute from Theorem \ref{thm:strongrandom} that
\[
4 \|P^t(\nu,\cdot) - \pi\|_\TV^2 \leq \|P^t(\nu,\cdot) - \pi\|_{L^2(\pi)}^2 \leq \sum_{n = 2t+1}^\infty \Prob_\nu(T > n) \leq \sum_{n = 2t+1}^\infty A(\nu) \rho^n.
\]
Summing the geometric series gives
\[
\sum_{n = 2t+1}^\infty A(\nu) \rho^n = \frac{A(\nu) \rho}{1 - \rho} \cdot \rho^{2t} = 4D^2 \rho^{2t}
\]
and so we have
\begin{equation}
\label{eq:nu-TV}
\|P^t(\nu,\cdot) - \pi\|_\TV \leq D \rho^t.
\end{equation}

Next we show that for any $x \in \X$,
\begin{equation} \label{eq:first-entrance}
\|P^t(x,\cdot) - \pi\|_\TV \leq \sum_{n=0}^t \Prob_x(T = n) \|P^{t-n}(\nu,\cdot) - \pi\|_\TV + \Prob_x(T > t).
\end{equation}
The proof of \eqref{eq:first-entrance} begins with the definition
\[
\|P^t(x,\cdot) - \pi\|_\TV = \sup_{S \subset \X} \Big[ P^t(x,S) - \pi(S) \Big],
\]
which is equivalent to our definition in Section \ref{sec:overview}. Because $T$ is a strong random time with measure $\nu$,
\begin{align*}
P^t(x,S) &= \sum_{n=0}^t \Prob_x(T = n) P^{t-n}(\nu,S) + \Prob_x(T > t) \Prob_x(X_t\in S \mid T > t), \\
\pi(S) &= \sum_{n=0}^t \Prob_x(T = n) \pi(S) + \Prob_x(T > t) \pi(S).
\end{align*}
Hence,
\[
P^t(x,S) - \pi(S) \leq \sum_{n=0}^t \Prob_x(T = n) \sup_{S'\subset\X} \Big[ P^{t-n}(\nu,S') - \pi(S') \Big] + \Prob_x(T > t).
\]
Taking the supremum over all $S \subset \X$ gives \eqref{eq:first-entrance}.

We now combine \eqref{eq:nu-TV} with \eqref{eq:first-entrance}:
\[
\|P^t(x,\cdot) - \pi\|_\TV \leq \sum_{n=0}^t \Prob_x(T = n) D \rho^{t-n} + \Prob_x(T > t).
\]
Summation by parts implies that
\[
\begin{split}
\sum_{n=0}^t \Prob_x(T = n) D \rho^{t-n} &= D \left[ \rho^t - \Prob_x(T > t) + \sum_{n=0}^{t-1} \Prob_x(T > n) (\rho^{t-n-1} - \rho^{t-n}) \right] \\
&\leq D \left[ \rho^t - \Prob_x(T > t) + \sum_{n=0}^{t-1} A(x) \rho^n \cdot \frac{1-\rho}{\rho}\rho^{t-n} \right] \\
&= D \left[ \rho^t - \Prob_x(T > t) + \frac{1-\rho}{\rho} A(x) t \rho^t \right].
\end{split}
\]
Therefore,
\[
\|P^t(x,\cdot) - \pi\|_\TV \leq D \rho^t + \frac{1-\rho}{\rho} DA(x) t \rho^t + (1-D) \Prob_x(T > t)
\]
and, since $1-D$ may be either positive or negative, we write
\[
(1-D) \Prob_x(T > t) \leq (1-D)_+ A(x) \rho^t.
\]
In conclusion,
\[
\|P^t(x,\cdot) - \pi\|_\TV \leq \left( \frac{1-\rho}{\rho} DA(x) t + (1-D)_+ A(x) + D \right) \rho^t. \qedhere
\]
\end{proof}

\section{Tail bound}
\label{sec:tail}

In this section we prove Theorem \ref{thm:tail-bound}, which sharpens Theorem 4.1(i) of \cite{RT99} and generalizes it to the case $m > 1$. Combined with the proofs of Theorems \ref{thm:strongrandom} and \ref{thm:TV-strong} in Section \ref{sec:proofs-strongrandom}, this finishes the proof of Theorem \ref{thm:TV-drift-minorize}.

To start, let $(X_t)$ be a Markov chain on $\X$ and fix $C \subset \X$. The hitting time of $C$ is
\[
\tau_C = \min\{ t \geq 0 : X_t \in C \}.
\]
If the chain never reaches $C$ then we take $\tau_C = \infty$.

Suppose that $(X_t)$ has a drift function $V(x)$ with respect to $C$. The following well-known lemma (see e.g.\ \cite[Lemma 2.2]{LT96}) says that for each $x \in \X$, the value of $V(x)$ bounds an exponential moment of $\tau_C$ for the chain started at $x$.

\begin{lemma}
\label{lemma:TC}
Let $(X_t)$ be a Markov chain on $\X$ with transition kernel $P$. Given $C \subset \X$, suppose that the function $V: \X \to [1,\infty)$ satisfies $PV(x) \leq \lambda V(x)$ for $x \notin C$, where $\lambda < 1$ is fixed. Then
\[
\E_x[\lambda^{-\tau_C}] \leq V(x) \qquad \text{for all } x \in \X.
\]
\end{lemma}

We note that the function $V_0(x) = \E_x[\lambda^{-\tau_C}]$ itself satisfies $PV_0(x) = \lambda V_0(x)$ for $x \notin C$ and is therefore the minimal drift function for the given $C$ and $\lambda$. It immediately follows from Lemma \ref{lemma:TC} that $\E_\mu[\lambda^{-\tau_C}] \leq \mu(V)$ for all probability measures $\mu$ on $\X$ with $\mu(V) < \infty$.

\begin{proof}
If $x \in C$, then $\E_x[\lambda^{-\tau_C}] = 1 \leq V(x)$.

Assume that $x \notin C$. For all $t \geq 0$, we compute
\[
\begin{split}
\E_x[V(X_t),\, \tau_C > t] &\geq \lambda^{-1}\E_x[PV(X_t),\, \tau_C > t] \\
&= \lambda^{-1}\E_x[V(X_{t+1}),\, \tau_C > t] \\
&= \lambda^{-1}\Big( \E_x[V(X_{t+1}),\, \tau_C > t+1] + \E_x[V(X_{t+1}),\, \tau_C = t+1] \Big) \\
&\geq \lambda^{-1}\E_x[V(X_{t+1}),\, \tau_C > t+1] + \lambda^{-1}\Prob_x(\tau_C = t+1).
\end{split}
\]
Observing that $\E_x[V(X_0),\, \tau_C > 0] = V(x)$, it follows by induction that
\begin{equation}
\label{eq:V-induction}
V(x) \geq \lambda^{-t} \E_x[V(X_t),\, \tau_C > t] + \sum_{n=1}^t \lambda^{-n} \Prob_x(\tau_C = n) \qquad \text{for all } t \geq 0.
\end{equation}
If we keep only the first term from the right side of \eqref{eq:V-induction}, we see that
\[
V(x) \geq \lambda^{-t} \Prob_x(\tau_C > t) \geq \lambda^{-t} \Prob_x(\tau_C = \infty)
\]
and sending $t \to \infty$ implies that $\Prob_x(\tau_C = \infty) = 0$. Now, we return to \eqref{eq:V-induction} and this time keep only the sum on the right side. Sending $t \to \infty$ gives
\[
V(x) \geq \sum_{n=1}^\infty \lambda^{-n} \Prob_x(\tau_C = n) = \E_x[\lambda^{-\tau_C}]. \qedhere
\]
\end{proof}

In the next lemma, we bound
\[
P^m(\mu,V) = \int_\X V(x) P^m(\mu,dx) = \int_\X P^m V(x) \mu(dx) = \E_\mu[V(X_m)]
\]
for measures $\mu$ supported on $C$. When $m = 1$, the assumption that $PV(x) \leq K$ for all $x \in C$ implies immediately that $P(\mu,V) \leq K$.

\begin{lemma}
\label{lemma:B-bound}
Let $(X_t)$ be a Markov chain on $\X$ with transition kernel $P$. Suppose that $(X_t)$ has a drift function $V(x)$ with respect to $C \subset \X$, with parameters $\lambda < 1$ and $K < \infty$. Fix $m \geq 1$. For any probability measure $\mu$ supported on $C$, we have
\[
P^m(\mu,V) \leq \frac{1 - \lambda^m}{1 - \lambda} (K - \lambda) + \lambda^m.
\]
\end{lemma}

The upper bound in Lemma \ref{lemma:B-bound} is exactly the formula for $B$ in \eqref{eq:B-def}, and it evaluates to $K$ when $m = 1$.

\begin{proof}
We aim to show that
\[
\E_\mu[V(X_m)] \leq \frac{1 - \lambda^m}{1 - \lambda} (K - \lambda) + \lambda^m.
\]
When $m = 1$ we have seen that this follows directly from the condition $PV(x) \leq K$ for all $x \in C$. In general, we use that $PV(x) \leq \lambda V(x) + (K - \lambda)$ (which is implied by the drift condition) to compute
\[
\E_\mu[V(X_{m+1})] = \E_\mu[PV(X_m)] \leq \lambda \E_\mu[V(X_m)] + (K-\lambda).
\]
The desired bound then follows by induction.
\end{proof}

\begin{proof}[Proof of Theorem \ref{thm:tail-bound}]
The first statement in Theorem \ref{thm:tail-bound} is that $\Prob_\mu(T < \infty) = 1$ for all probability measures $\mu$ on $\X$. This follows from the bound
\begin{equation}
\label{eq:tail-bound-formula}
\Prob_\mu(T > t) \leq \mu(V)^r \rho^{t+1-m} \qquad \text{for all } t \geq 0
\end{equation}
that we will obtain when $\mu(V) < \infty$: from \eqref{eq:tail-bound-formula} we have $\Prob_x(T > t) \leq V(x)^r \rho^{t+1-m}$ for all $x \in \X$, hence $\Prob_x(T < \infty) = 1$, and then for any $\mu$ we can write
\[
\Prob_\mu(T < \infty) = \int_\X \Prob_x(T < \infty) \mu(dx) = 1.
\]
It therefore suffices to prove \eqref{eq:tail-bound-formula} along with the upper bound
\begin{equation}
\label{eq:nu-V}
\nu(V) \leq \frac{B - (1-\e)}{\e}.
\end{equation}

We first consider the case $\e = 1$. Here we have $T = \tau_C + m$ and $\nu(\cdot) = P^m(x,\cdot)$ for every $x \in C$. The bound \eqref{eq:nu-V} is simply $\nu(V) \leq B$, which holds by Lemma \ref{lemma:B-bound}. To verify \eqref{eq:tail-bound-formula}, fix a measure $\mu$ with $\mu(V) < \infty$. Using Markov's inequality followed by Lemma \ref{lemma:TC}, we obtain \eqref{eq:tail-bound-formula} with $\rho = \lambda$ and $r = 1$:
\[
\Prob_\mu(T>t) \leq \lambda^{t+1}\E_\mu[\lambda^{-T}] = \lambda^{t+1-m}\E_\mu[\lambda^{-\tau_C}] \leq \mu(V)\lambda^{t+1-m}.
\]

Assume now that $\e < 1$. For any probability measure $\mu$ supported on $C$, write
\begin{equation}
\label{eq:mu-bar-def}
P^m(\mu,\cdot) = \e \nu(\cdot) + (1-\e) \mubar(\cdot).
\end{equation}
The minorization property implies that $\mubar$ is a probability measure. By Lemma \ref{lemma:B-bound},
\begin{equation}
\label{eq:measure-decomp}
\e \nu(V) + (1-\e) \mubar(V) = P^m(\mu,V) \leq B.
\end{equation}
When we combine \eqref{eq:measure-decomp} with the lower bound $\mubar(V) \geq 1$, we get \eqref{eq:nu-V}. When we instead combine \eqref{eq:measure-decomp} with the lower bound $\nu(V) \geq 1$, we get
\begin{equation}
\label{eq:mubar-V}
\mubar(V) \leq \frac{B-\e}{1-\e}.
\end{equation}

Let $\D$ be the set of all measures $\mubar$ that appear in \eqref{eq:mu-bar-def} when $\mu$ varies over all the probability measures supported on $C$. We would like to prove by induction that
\begin{equation}
\label{eq:induct-goal}
\sup_{\mubar \in \D} \Prob_\mubar(T > t) \leq \alpha \rho^t \qquad \text{for all } t \geq 0,
\end{equation}
for constants $\alpha < \infty$ and $\rho < 1$ whose values we will determine later. (In the end, the value of $\rho$ will be given by \eqref{eq:rho-def}.)

Suppose that $t \geq m$ and we have already proved \eqref{eq:induct-goal} for all $t' < t$. Write
\[
\Prob_\mubar(T > t) = \sum_{s=0}^{t-m} \Prob_\mubar(T > t \mid \tau_C = s) \Prob_\mubar(\tau_C = s) + \Prob_\mubar(\tau_C > t-m).
\]
For $s \leq t-m$, define $\mubar_s \in \D$ to satisfy
\[
\Prob_\mubar(X_{s+m}\in\cdot \mid \tau_C = s) = \e\nu(\cdot) + (1-\e)\mubar_s(\cdot).
\]
By the construction of $T$ and the inductive hypothesis,
\[
\Prob_\mubar(T > t \mid \tau_C = s) = (1-\e) \Prob_{\mubar_s}(T > t-s-m) \leq (1-\e) \alpha \rho^{t-s-m}.
\]
Therefore,
\[
\Prob_\mubar(T > t) \leq \sum_{s=0}^{t-m} (1-\e)\alpha \rho^{t-s-m} \Prob_\mubar(\tau_C = s) + \Prob_\mubar(\tau_C > t-m).
\]
We would like to argue next that
\[
\Prob_\mubar(\tau_C > t-m) = \sum_{s=t-m+1}^\infty \Prob_\mubar(\tau_C = s) \leq \sum_{s=t-m+1}^\infty (1-\e)\alpha \rho^{t-s-m} \Prob_\mubar(\tau_C = s)
\]
and this inequality will hold as long as
\begin{equation}
\label{eq:induct-cond1}
(1-\e) \alpha \rho^{-1} \geq 1.
\end{equation}
If \eqref{eq:induct-cond1} is true, then
\[
\Prob_\mubar(T > t) \leq \sum_{s=0}^\infty (1-\e)\alpha \rho^{t-s-m} \Prob_\mubar(\tau_C = s) = (1-\e) \alpha \E_\mubar[\rho^{-\tau_C - m}] \cdot \rho^t.
\]
If in addition we have
\begin{equation}
\label{eq:induct-cond2}
(1-\e) \E_\mubar[\rho^{-\tau_C - m}] \leq 1
\end{equation}
then we get $\Prob_\mubar(T > t) \leq \alpha \rho^t$ and the induction is complete.

Assume that \eqref{eq:induct-cond1} and \eqref{eq:induct-cond2} hold. By combining them, we see that
\[
\alpha \rho^{-1} \geq \E_\mubar[\rho^{-\tau_C - m}] \geq \rho^{-m}.
\]
The base case $t \leq m-1$ of \eqref{eq:induct-goal} is proved by observing that $\alpha \rho^t \geq \alpha \rho^{m-1} \geq 1$. Therefore, for $\alpha$ and $\rho < 1$ satisfying \eqref{eq:induct-cond1} and \eqref{eq:induct-cond2}, we have finished the inductive proof of \eqref{eq:induct-goal}. In fact, once we have found $\rho < 1$ satisfying \eqref{eq:induct-cond2}, we can let $\alpha = \rho / (1-\e)$ so that \eqref{eq:induct-cond1} holds.

For values of $\rho$ less than $\lambda$, we cannot bound the exponential moment in \eqref{eq:induct-cond2}. For $\lambda \leq \rho \leq 1$, we can set $r = \log \rho / \log \lambda \in [0,1]$ and write
\[
\E_\mubar[\rho^{-\tau_C - m}] = \rho^{-m} \E_\mubar[\lambda^{-\tau_C r}] \leq \rho^{-m} \E_\mubar[\lambda^{-\tau_C}]^r
\]
using Jensen's inequality (the function $x \mapsto x^r$ is concave). Then, Lemma \ref{lemma:TC} and \eqref{eq:mubar-V} imply that
\begin{equation}
\label{eq:rho-function}
\E_\mubar[\rho^{-\tau_C - m}] \leq \rho^{-m} \left( \frac{B-\e}{1-\e} \right)^r.
\end{equation}
The right side of \eqref{eq:rho-function} is decreasing in $\rho$ and evaluates to $1$ when $\rho = 1$. Setting this quantity equal to $1/(1-\e)$ and solving for $\rho$ yields the solution
\[
\rho_0 = \exp\left( \frac{-\log (1-\e) \log \lambda}{-m \log \lambda + \log(B-\e) - \log(1-\e)} \right) < 1.
\]
Thus, if we set $\rho = \lambda \vee \rho_0$ (matching the definition in \eqref{eq:rho-def}) then $\rho < 1$ and we have proved \eqref{eq:induct-cond2}. Letting $\alpha = \rho/(1-\e)$, we conclude that
\begin{equation}
\label{eq:induct-done}
\sup_{\mubar \in \D} \Prob_\mubar(T > t) \leq \frac{1}{1-\e} \rho^{t+1} \qquad \text{for all } t \geq 0.
\end{equation}

We are now ready to prove \eqref{eq:tail-bound-formula}. Let $\mu$ be a probability measure on $\X$ with $\mu(V) < \infty$. For $t \leq m-1$, the right side of \eqref{eq:tail-bound-formula} is at least $1$ and so the statement is trivial. For $t \geq m$, we repeat the argument from the induction. Write
\[
\Prob_\mu(T > t) = \sum_{s=0}^{t-m} \Prob_\mu(T > t \mid \tau_C = s) \Prob_\mu(\tau_C = s) + \Prob_\mu(\tau_C > t-m).
\]
For $s \leq t-m$, define $\mu_s \in \D$ to satisfy
\[
\Prob_\mu(X_{s+m}\in\cdot \mid \tau_C = s) = \e\nu(\cdot) + (1-\e)\mu_s(\cdot).
\]
Then, using \eqref{eq:induct-done},
\[
\Prob_\mu(T > t \mid \tau_C = s) = (1-\e) \Prob_{\mu_s}(T > t-s-m) \leq \rho^{t-s-m+1}
\]
and it follows that
\[
\Prob_\mu(T > t) \leq \sum_{s=0}^{t-m} \rho^{t-s-m+1} \Prob_\mu(\tau_C = s) + \Prob_\mu(\tau_C > t-m).
\]
Since
\[
\Prob_\mu(\tau_C > t-m) = \sum_{s=t-m+1}^\infty \Prob_\mu(\tau_C = s) \leq \sum_{s=t-m+1}^\infty \rho^{t-s-m+1} \Prob_\mu(\tau_C = s),
\]
we have
\[
\Prob_\mu(T > t) \leq \sum_{s=0}^\infty \rho^{t-s-m+1} \Prob_\mu(\tau_C = s) = \rho^{t-m+1} \E_\mu[\rho^{-\tau_C}].
\]
Again, Jensen's inequality and Lemma \ref{lemma:TC} give
\[
\E_\mu[\rho^{-\tau_C}] = \E_\mu[\lambda^{-\tau_C r}] \leq \E_\mu[\lambda^{-\tau_C}]^r \leq \mu(V)^r
\]
and we conclude as desired that
\[
\Prob_\mu(T > t) \leq \mu(V)^r \rho^{t-m+1}. \qedhere
\]
\end{proof}

\section{$V$-norm convergence}
\label{sec:V-norm}

Let $(X_t)$ be a Markov chain on $\X$ with transition kernel $P$ and stationary distribution $\pi$. Suppose that $(X_t)$ has a drift function $V(x)$ with respect to a small set. It is a well-established principle (see e.g.\ \cite{MT94}) that any upper bound on $\|P^t(x,\cdot) - \pi\|_\TV$ can easily be strengthened to an upper bound on the $V$-norm distance
\[
\|P^t(x,\cdot) - \pi\|_V = \sup_{|f| \leq V} |P^t f(x) - \pi(f)|.
\]

The goal of this section is to carry out the strengthening process for Theorem \ref{thm:TV-drift-minorize}. We will prove the following result. Note that the bound \eqref{eq:TV-reprise} below is simply a restatement of Theorem \ref{thm:TV-drift-minorize}.

\begin{thm}
\label{thm:V-norm}
Let $(X_t)$ be a Markov chain on $\X$ with transition kernel $P$. Assume that $(X_t)$ is reversible with respect to the probability measure $\pi$ and that it has nonnegative eigenvalues. Also assume that $(X_t)$ has a drift function with respect to a small set $C \subset \X$. Let $V(x)$, $\lambda < 1$, $K < \infty$ be the data associated with the drift function, and let $\nu$, $m \geq 1$, $\e > 0$ be the data associated with the small set.

Define
\[
B = \frac{1 - \lambda^m}{1 - \lambda} (K - \lambda) + \lambda^m.
\]
If $\e = 1$, then set $\rho = \lambda$ and $r = 1$. If $\e < 1$, then set
\begin{align*}
\rho &= \lambda \vee \exp\left( \frac{-\log(1-\e) \log \lambda}{-m \log \lambda + \log(B - \e) - \log(1-\e)} \right), \\
r &= \frac{\log \rho}{\log \lambda}.
\end{align*}
Let
\[
D = \frac{1}{2} \sqrt{ \left[ \frac{B - (1-\e)}{\e} \right]^r \frac{\rho^{2-m}}{1-\rho} }
\]
and define the functions
\begin{align*}
F_0(x) &= (1-D)_+ \rho^{1-m} V(x)^r + D, \\
F_1(x) &= (1-\rho) \rho^{-m} D V(x)^r.
\end{align*}
Given these definitions, we have
\begin{equation}
\label{eq:TV-reprise}
\|P^t(x,\cdot) - \pi\|_\TV \leq [F_1(x)t + F_0(x)] \rho^t \qquad \text{for all $x \in \X$, $t \geq 0$.}
\end{equation}

Next, let
\[
G_0(x) = V(x) + \frac{K-\lambda}{1-\lambda}.
\]
If $\rho = \lambda$, define
\begin{align*}
G_1(x) &= K \lambda^{-1} [2F_0(x) - F_1(x)], \\
G_2(x) &= K \lambda^{-1} F_1(x),
\end{align*}
while if $\rho > \lambda$, define instead
\begin{align*}
H_0(x) &= 2K \left[ \frac{F_0(x)}{\rho - \lambda} - \frac{\rho F_1(x)}{(\rho - \lambda)^2} \right], \\
H_1(x) &= \frac{2K F_1(x)}{\rho - \lambda}.
\end{align*}
We then have the following bounds. If $\rho = \lambda$, then
\[
\|P^t(x,\cdot) - \pi\|_V \leq [G_2(x) t^2 + G_1(x) t + G_0(x)] \lambda^t \qquad \text{for all $x \in \X$, $t \geq 0$.}
\]
If $\rho > \lambda$, then
\[
\|P^t(x,\cdot) - \pi\|_V \leq [H_1(x)t + H_0(x)] \rho^t + [G_0(x) - H_0(x)] \lambda^t \qquad \text{for all $x \in \X$, $t \geq 0$.}
\]
\end{thm}

The first step in the proof of Theorem \ref{thm:V-norm} is a standard upper bound on $\pi(V)$.

\begin{lemma}
\label{lemma:pi-V}
Let $(X_t)$ be a Markov chain on $\X$ with transition kernel $P$ and stationary distribution $\pi$. Suppose that $(X_t)$ has a drift function $V(x)$ with respect to $C \subset \X$, with parameters $\lambda < 1$ and $K < \infty$. Then,
\begin{equation}
\label{eq:pi-V}
\pi(V) \leq \frac{K-\lambda}{1-\lambda} \pi(C).
\end{equation}
\end{lemma}

One might attempt to prove Lemma \ref{lemma:pi-V} by using stationarity of $\pi$ to write
\[
\pi(V) = \int_\X PV(x) \pi(dx).
\]
The drift condition then implies that
\[
\pi(V) \leq \lambda \pi(V) + (K-\lambda) \pi(C),
\]
which is equivalent to \eqref{eq:pi-V} as long as $\pi(V) < \infty$. Thus, showing that $\pi(V) < \infty$ is the main difficulty in the proof of Lemma \ref{lemma:pi-V}. In \cite[Theorem 14.3.7]{MT93}, Meyn and Tweedie derive a more general version of Lemma \ref{lemma:pi-V} using ergodic properties of the chain $(X_t)$. Our proof here is direct and shows finiteness by a truncation argument.

\begin{proof}
Fix a positive integer $N$. Let $S_N = \{x\in\X : V(x) \leq N\}$, and let $U_N = \X\setminus S_N$. We first observe that
\[
\begin{split}
\Prob_\pi(X_0 \in U_N,\, X_1 \in S_N) &= \Prob_\pi(X_1 \in S_N) - \Prob_\pi(X_0 \in S_N,\, X_1 \in S_N) \\
&= \Prob_\pi(X_0 \in S_N) - \Prob_\pi(X_0 \in S_N,\, X_1 \in S_N) \\
&= \Prob_\pi(X_0 \in S_N,\, X_1 \in U_N).
\end{split}
\]
Therefore,
\[
\begin{split}
\E_\pi[V(X_1),\, X_0 \in U_N,\, X_1 \in S_N] &\leq N \Prob_\pi(X_0 \in U_N,\, X_1 \in S_N) \\
&= N \Prob_\pi(X_0 \in S_N,\, X_1 \in U_N) \\
&\leq \E_\pi[V(X_1),\, X_0 \in S_N,\, X_1 \in U_N].
\end{split}
\]
It follows that
\begin{multline*}
\E_\pi[V(X_0),\, X_0 \in S_N] \\
\begin{aligned}
&= \E_\pi[V(X_1),\, X_1 \in S_N] \\
&= \E_\pi[V(X_1),\, X_0 \in S_N,\, X_1 \in S_N] + \E_\pi[V(X_1),\, X_0 \in U_N,\, X_1 \in S_N] \\
&\leq \E_\pi[V(X_1),\, X_0 \in S_N,\, X_1 \in S_N] + \E_\pi[V(X_1),\, X_0 \in S_N,\, X_1 \in U_N] \\
&= \E_\pi[V(X_1),\, X_0 \in S_N] \\
&= \E_\pi[PV(X_0),\, X_0 \in S_N].
\end{aligned}
\end{multline*}
The drift condition implies that $PV(x) \leq \lambda V(x) + (K-\lambda) \1\{x \in C\}$. Thus,
\[
\begin{split}
\E_\pi[V(X_0),\, X_0 \in S_N] &\leq \E_\pi[PV(X_0),\, X_0 \in S_N] \\
&\leq \lambda \E_\pi[V(X_0),\, X_0 \in S_N] + (K-\lambda) \pi(S_N \cap C).
\end{split}
\]
Since $\E_\pi[V(X_0),\, X_0 \in S_N] \leq N$, we can subtract to obtain
\[
(1-\lambda) \E_\pi[V(X_0),\, X_0 \in S_N] \leq (K-\lambda) \pi(S_N \cap C).
\]
Finally, by monotone convergence,
\[
(1-\lambda) \pi(V) = \lim_{N \to \infty} (1-\lambda) \E_\pi[V(X_0),\, X_0 \in S_N] \leq (K-\lambda) \pi(C). \qedhere
\]
\end{proof}

The following lemma will allow us to go from total variation convergence to convergence in $V$-norm.

\begin{lemma} \label{lemma:TV-to-V-norm}
Let $(X_t)$ be a Markov chain on $\X$ with transition kernel $P$. Suppose that $(X_t)$ has a drift function $V(x)$ with respect to $C \subset \X$, with parameters $\lambda < 1$ and $K < \infty$. Then, for any probability measures $\mu,\mu'$ on $\X$ with $\mu(V), \mu'(V) < \infty$,
\begin{multline*}
\|P^t(\mu,\cdot)-P^t(\mu',\cdot)\|_V \\
\leq 2K \sum_{n=1}^t \lambda^{n-1} \|P^{t-n}(\mu,\cdot)-P^{t-n}(\mu',\cdot)\|_\TV + [\mu(V)+\mu'(V)] \lambda^t \qquad \text{for all } t \geq 0.
\end{multline*}
\end{lemma}

\begin{proof}
We first show by induction that
\begin{equation}
\label{eq:mu-moment}
\E_\mu[V(X_t),\, \tau_C\geq t] \leq \mu(V)\lambda^t \qquad \text{for all } t \geq 0.
\end{equation}
The base case $t = 0$ is trivial. For $t\geq 1$,
\[
\begin{split}
\E_\mu[V(X_t),\, \tau_C\geq t] &= \E_\mu[PV(X_{t-1}),\, \tau_C\geq t] \\
&\leq \lambda \E_\mu[V(X_{t-1}),\, \tau_C\geq t] \\
&\leq \lambda \E_\mu[V(X_{t-1}),\, \tau_C\geq t-1],
\end{split}
\]
which finishes the inductive proof. If we let
\[
\tau_C^+ = \min\{t \geq 1 : X_t \in C\},
\]
then the same argument shows that every $x\in C$ satisfies
\begin{equation}
\label{eq:C-moment}
\E_x[V(X_t),\, \tau_C^+ \geq t] \leq K\lambda^{t-1} \qquad \text{for all } t \geq 1.
\end{equation}

Fix $t \geq 0$. For $s < t$, let $E_s$ be the event that $X_s \in C$ and $X_k \notin C$ for all $s < k < t$. We have
\begin{equation}
\label{eq:V-norm-specific}
\|P^t(\mu,\cdot)-P^t(\mu',\cdot)\|_V = \sup_{|f|\leq V} |P^t(\mu,f)-P^t(\mu',f)|,
\end{equation}
where
\begin{equation}
\label{eq:mu-f-decomp}
P^t(\mu,f) = \E_\mu[f(X_t)] = \sum_{n=1}^t \E_\mu[f(X_t),\, E_{t-n}] + \E_\mu[f(X_t),\, \tau_C \geq t]
\end{equation}
and
\begin{equation}
\label{eq:C-integral}
\E_\mu[f(X_t),\, E_{t-n}] = \int_C \E_x[f(X_n),\, \tau_C^+ \geq n] P^{t-n}(\mu,dx).
\end{equation}

We will use the fact that for probability measures $\eta,\eta'$ on $\X$,
\begin{equation}
\label{eq:TV-identity}
\|\eta - \eta'\|_\TV = \frac{1}{2} \int_\X |\eta - \eta'|(dx)
\end{equation}
where $|\eta - \eta'|$ is the variation of the signed measure $\eta - \eta'$ \cite{C13}. Let $|f| \leq V$. For $1 \leq n \leq t$, we use \eqref{eq:C-integral}, \eqref{eq:C-moment}, \eqref{eq:TV-identity} to compute
\begin{multline*}
\Big| \E_\mu[f(X_t),\, E_{t-n}] - \E_{\mu'}[f(X_t),\, E_{t-n}] \Big| \\
\begin{aligned}
&\leq \int_C \E_x[V(X_n),\, \tau_C^+ \geq n] \, |P^{t-n}(\mu,\cdot) - P^{t-n}(\mu',\cdot)|(dx) \\
&\leq K \lambda^{n-1} \int_C |P^{t-n}(\mu,\cdot) - P^{t-n}(\mu',\cdot)|(dx) \\
&\leq 2K \lambda^{n-1} \|P^{t-n}(\mu,\cdot) - P^{t-n}(\mu',\cdot)\|_\TV.
\end{aligned}
\end{multline*}
In addition, \eqref{eq:mu-moment} implies that
\[
\Big| \E_\mu[f(X_t),\, \tau_C \geq t] - \E_{\mu'}[f(X_t),\, \tau_C \geq t] \Big| \leq [\mu(V) + \mu'(V)] \lambda^t.
\]
Using \eqref{eq:mu-f-decomp}, $|P^t(\mu,f) - P^t(\mu',f)|$ is bounded above by
\[
2K \sum_{n=1}^t \lambda^{n-1} \|P^{t-n}(\mu,\cdot) - P^{t-n}(\mu',\cdot)\|_\TV + [\mu(V) + \mu'(V)] \lambda^t,
\]
so by \eqref{eq:V-norm-specific} the proof is complete.
\end{proof}

\begin{proof}[Proof of Theorem \ref{thm:V-norm}]
The idea is to combine the total variation bound \eqref{eq:TV-reprise}, which we already proved as Theorem \ref{thm:TV-drift-minorize}, with Lemma \ref{lemma:TV-to-V-norm}. In Lemma \ref{lemma:TV-to-V-norm} we take $\mu$ to be the $\delta$-measure at $x$ and $\mu' = \pi$. We compute
\[
\begin{split}
\|P^t(x,\cdot) - \pi\|_V &\leq 2K \sum_{n=1}^t \lambda^{n-1} \|P^{t-n}(x,\cdot) - \pi\|_\TV + [V(x) + \pi(V)] \lambda^t \\
&\leq 2K \sum_{n=1}^t \lambda^{n-1} [F_1(x)(t-n) + F_0(x)] \rho^{t-n} + \left[ V(x) + \frac{K-\lambda}{1-\lambda} \right] \lambda^t
\end{split}
\]
using \eqref{eq:TV-reprise} and Lemma \ref{lemma:pi-V}. The last term is $G_0(x) \lambda^t$. For the rest, write
\begin{multline*}
2K \sum_{n=1}^t \lambda^{n-1} [F_1(x)(t-n) + F_0(x)] \rho^{t-n} \\
= 2K F_1(x) \sum_{n=1}^t \lambda^{n-1} (t-n) \rho^{t-n} + 2K F_0(x) \sum_{n=1}^t \lambda^{n-1} \rho^{t-n}.
\end{multline*}
Routine computations show that when $\rho = \lambda$ we have
\[
\sum_{n=1}^t \lambda^{n-1} (t-n) \rho^{t-n} = \frac{t^2 - t}{2} \lambda^{t-1}, \qquad \sum_{n=1}^t \lambda^{n-1} \rho^{t-n} = t \lambda^{t-1}
\]
and when $\rho > \lambda$ we have
\[
\sum_{n=1}^t \lambda^{n-1} (t-n) \rho^{t-n} = \frac{t \rho^t}{\rho - \lambda} - \frac{\rho(\rho^t - \lambda^t)}{(\rho - \lambda)^2}, \qquad \sum_{n=1}^t \lambda^{n-1} \rho^{t-n} = \frac{\rho^t - \lambda^t}{\rho - \lambda}.
\]
Thus, when $\rho = \lambda$ we obtain
\[
\begin{split}
\|P^t(x,\cdot) - \pi\|_V &\leq 2KF_1(x) \cdot \frac{t^2-t}{2} \lambda^{t-1} + 2KF_0(x) \cdot t \lambda^{t-1} + G_0(x) \lambda^t \\
&= [G_2(x) t^2 + G_1(x) t + G_0(x)] \lambda^t
\end{split}
\]
and when $\rho > \lambda$ we obtain
\[
\begin{split}
\|P^t(x,\cdot) - \pi\|_V &\leq 2KF_1(x) \left[ \frac{t \rho^t}{\rho - \lambda} - \frac{\rho(\rho^t - \lambda^t)}{(\rho - \lambda)^2} \right] + 2KF_0(x) \cdot \frac{\rho^t - \lambda^t}{\rho - \lambda} + G_0(x) \lambda^t \\
&= [H_1(x)t + H_0(x)] \rho^t + [G_0(x) - H_0(x)] \lambda^t,
\end{split}
\]
which are the desired bounds.
\end{proof}

\section{Example: Nuclear pump Gibbs sampler}
\label{sec:example}

In this section we analyze a Markov chain previously considered by \cite{GS90,T94,MTY95,R95}. It is a single-component chain of a two-component deterministic scan Gibbs sampler, and as such is reversible with nonnegative eigenvalues \cite{B05}.

Our Gibbs sampler was proposed by Gelfand and Smith \cite{GS90} as a means of sampling from the posterior distribution associated with a Bayesian hierarchical model for the failure rate of pumps in a nuclear power plant. See \cite{T94} for a description of the model and a derivation of the transition rule below. We use the convention that a Gamma distribution $G(a,b)$ has density function $G_{a,b}(x) = \frac{b^a}{\Gamma(a)} x^{a-1} e^{-bx} \1\{x \geq 0\}$.

Let $(s_1,\ldots,s_{10};\, t_1,\ldots,t_{10})$ be the numerical data from \cite[Table 3]{GS90}. The Gibbs sampler $(\beta_t, S_t)$ on $\R^2$ is defined as follows. From the state $(\beta_t, S_t)$, sample $\beta_{t+1} \sim G(18.03, 1 + S_t)$. Then, independently for each $1 \leq j \leq 10$, sample $\theta_{t+1}^{(j)} \sim G(1.802 + s_j, \beta_{t+1} + t_j)$. Finally, set $S_{t+1} = \sum_{j=1}^{10} \theta_{t+1}^{(j)}$.\footnote{It is easily seen that this definition is equivalent to the $11$-component Gibbs sampler on $(\beta, \theta^{(1)}, \ldots, \theta^{(10)})$ which updates each coordinate in sequence. That is the definition provided in \cite{GS90,T94,MTY95,R95}.} The update order is
\[
\cdots \rightarrow (\beta_t, S_t) \rightarrow (\beta_{t+1}, S_t) \rightarrow (\beta_{t+1}, S_{t+1}) \rightarrow (\beta_{t+2}, S_{t+1}) \rightarrow \cdots
\]
From this description it follows that $(S_t)$ is itself a Markov chain which converges at the same rate as the full Gibbs sampler $(\beta_t, S_t)$. It is not hard to verify that any single-component chain defined from a two-component deterministic scan Gibbs sampler in this manner must be reversible with nonnegative eigenvalues.

Let the chain $(S_t)$ have transition kernel $P$ and stationary distribution $\pi$. Rosenthal \cite{R95} used the bivariate drift method to compute a numerical bound of the form \eqref{eq:geom-ergodic} on the distance from stationarity after $t$ steps. This bound is tightest when the chain is started at $S_0 = 6.5$. If we let
\[
\tau_\TV(0.01) = \min\{ t \geq 0 : \|P^t(6.5,\cdot) - \pi\|_\TV \leq 0.01 \},
\]
then \cite[Theorem 11]{R95} implies that $\tau_\TV(0.01) \leq 192$.

In Lemma \ref{lemma:explicit-parameters} we find explicit drift and minorization data for $(S_t)$ with the drift function $V(x) = 1 + (x - 6.5)^2$. This can be plugged into Theorems \ref{thm:TV-drift-minorize} and \ref{thm:V-norm} as well as \cite[Theorems 1.1 and 1.3]{B05}. Recall from Sections \ref{sec:related} and \ref{sec:nonreversible} that \cite[Theorem 1.3]{B05} also requires the chain to be reversible with nonnegative eigenvalues, while \cite[Theorem 1.1]{B05} does not require reversibility. We define
\[
\tau_V(0.02) = \min\{ t \geq 0 : \|P^t(6.5,\cdot) - \pi\|_V \leq 0.02 \},
\]
so that $\tau_\TV(0.01) \leq \tau_V(0.02)$, and compare the resulting upper bounds on $\tau_\TV(0.01)$ and $\tau_V(0.02)$ in Table \ref{table:bounds}.

\begin{table}
\centering
\begin{tabular}{| l | r | r |}
\multicolumn{1}{c}{Method} & \multicolumn{1}{c}{$\tau_\TV(0.01)$} & \multicolumn{1}{c}{$\tau_V(0.02)$} \\
\hline
\cite[Theorem 11]{R95} & $192$ & --- \\
\hline
\cite[Theorem 1.1]{B05} & --- & $1.0 \cdot 10^7$ \\
\hline
\cite[Theorem 1.3]{B05} & --- & $212$ \\
\hline
Theorem \ref{thm:TV-drift-minorize} & 83 & --- \\
\hline
Theorem \ref{thm:V-norm} & --- & 111 \\
\hline
\end{tabular}
\vspace{4pt}
\caption{Upper bounds on $\tau_\TV(0.01)$ and $\tau_V(0.02)$ by various methods.}
\label{table:bounds}
\end{table}

The bound of \cite[Theorem 1.1]{B05} is so large as to be impractical. The others are all relatively similar, with Theorems \ref{thm:TV-drift-minorize} and \ref{thm:V-norm} better by about a factor of $2$. None of these upper bounds is close to being sharp: a non-rigorous argument in \cite[Ch.~5]{J16} strongly indicates that $\tau_\TV(0.01) = 2$. This reinforces the principle that the method of drift and minorization only captures actual convergence rates for a limited class of Markov chains such as the monotone chains considered by \cite{LT96}. In all other circumstances, the best one can hope for is non-sharp bounds that are still small enough to be useful.

We now verify the drift and minorization condition for $(S_t)$. By design, we closely follow the proof of \cite[Theorem 11]{R95} so as to be sure that our improvements in convergence bounds are due to theoretical considerations rather than better estimates for this particular example. An expanded and illustrated version of the proof below can be found in \cite[Ch.~5]{J16}.

\begin{lemma}
\label{lemma:explicit-parameters}
The Markov chain $(S_t)$ has a drift function $V(x) = 1 + (x - 6.5)^2$ with respect to the small set $C = [4.74,8.50]$. The drift parameters are $\lambda = 0.61$, $K = 3.05$ and the minorization parameters are $m = 1$, $\e = 0.287$.
\end{lemma}

\begin{proof}
Rosenthal \cite{R95} observes that the mean of the stationary distribution $\pi$ for $(S_t)$ is roughly $6.5$, and for this reason chooses the bivariate drift function $W(x,y) = 1 + (x-6.5)^2 + (y-6.5)^2$. We use the corresponding univariate drift function $V(x) = 1 + (x-6.5)^2$.

For $\lambda = 0.01,0.02,\ldots,0.99$, we perform the following procedure.
\begin{enumerate}[label=\arabic*.]
\item Compute the function $PV(x)$ numerically and let $C$ be the set of $x$-values for which $PV(x) > \lambda V(x)$. Set $K = \sup_{x \in C} PV(x)$.
\item Take $m = 1$ in the minorization. Compute the minorization constant $\e$ by the same method used in the proof of \cite[Theorem 11]{R95}.
\item Use the formula \eqref{eq:rho-def} to compute the convergence rate $\rho$.
\end{enumerate}
In this way, for each possible value of $\lambda$ we obtain an upper bound $\rho = \rho(\lambda)$ on the convergence rate of the chain. The smallest of these $\rho$ values is $\rho(0.61) = 0.914$. Under the choice $\lambda = 0.61$, the set $C$ and the values of $K,\e$ are as in the statement of Lemma \ref{lemma:explicit-parameters}.
\end{proof}

\footnotesize{
\bibliographystyle{abbrvurl}
\bibliography{strongrandom}

\begin{thebibliography}{10}

\bibitem{AD86}
D.~Aldous and P.~Diaconis.
\newblock Shuffling cards and stopping times.
\newblock {\em Amer. Math. Monthly}, 93(5):333--348, 1986.
\newblock \href {https://doi.org/10.2307/2323590} {\path{doi:10.2307/2323590}}.

\bibitem{AD87}
D.~Aldous and P.~Diaconis.
\newblock Strong uniform times and finite random walks.
\newblock {\em Adv. in Appl. Math.}, 8(1):69--97, 1987.
\newblock \href {https://doi.org/10.1016/0196-8858(87)90006-6}
  {\path{doi:10.1016/0196-8858(87)90006-6}}.

\bibitem{AN78}
K.~B. Athreya and P.~Ney.
\newblock A new approach to the limit theory of recurrent {M}arkov chains.
\newblock {\em Trans. Amer. Math. Soc.}, 245:493--501, 1978.
\newblock \href {https://doi.org/10.2307/1998882} {\path{doi:10.2307/1998882}}.

\bibitem{B05}
P.~H. Baxendale.
\newblock Renewal theory and computable convergence rates for geometrically
  ergodic {M}arkov chains.
\newblock {\em Ann. Appl. Probab.}, 15(1B):700--738, 2005.
\newblock \href {https://doi.org/10.1214/105051604000000710}
  {\path{doi:10.1214/105051604000000710}}.

\bibitem{B13}
W.~Bednorz.
\newblock The {K}endall theorem and its application to the geometric ergodicity
  of {M}arkov chains.
\newblock {\em Applicationes Mathematicae}, 40(2):129--165, 2013.
\newblock URL: \url{http://eudml.org/doc/279924}.

\bibitem{BLL08}
W.~Bednorz, K.~{\L}atuszy\'{n}ski, and R.~Lata{\l}a.
\newblock A regeneration proof of the central limit theorem for uniformly
  ergodic {M}arkov chains.
\newblock {\em Electron. Commun. Probab.}, 13:85--98, 2008.
\newblock \href {https://doi.org/10.1214/ECP.v13-1354}
  {\path{doi:10.1214/ECP.v13-1354}}.

\bibitem{C13}
D.~L. Cohn.
\newblock {\em Measure theory}.
\newblock Birkh\"{a}user Advanced Texts: Basler Lehrb\"{u}cher. [Birkh\"{a}user
  Advanced Texts: Basel Textbooks]. Birkh\"{a}user/Springer, New York, second
  edition, 2013.
\newblock \href {https://doi.org/10.1007/978-1-4614-6956-8}
  {\path{doi:10.1007/978-1-4614-6956-8}}.

\bibitem{DFMS04}
R.~Douc, G.~Fort, E.~Moulines, and P.~Soulier.
\newblock Practical drift conditions for subgeometric rates of convergence.
\newblock {\em Ann. Appl. Probab.}, 14(3):1353--1377, 2004.
\newblock \href {https://doi.org/10.1214/105051604000000323}
  {\path{doi:10.1214/105051604000000323}}.

\bibitem{DMS07}
R.~Douc, E.~Moulines, and P.~Soulier.
\newblock Computable convergence rates for sub-geometric ergodic {M}arkov
  chains.
\newblock {\em Bernoulli}, 13(3):831--848, 2007.
\newblock \href {https://doi.org/10.3150/07-BEJ5162}
  {\path{doi:10.3150/07-BEJ5162}}.

\bibitem{F02}
G.~Fort.
\newblock Computable bounds for {V}-geometric ergodicity of {M}arkov transition
  kernels.
\newblock Rapport de Recherche, Univ. J. Fourier, RR 1047-M, 2002.
\newblock URL:
  \url{http://math.univ-toulouse.fr/~gfort/Preprints/fort:2002.pdf}.

\bibitem{FM03}
G.~Fort and E.~Moulines.
\newblock Polynomial ergodicity of {M}arkov transition kernels.
\newblock {\em Stochastic Process. Appl.}, 103(1):57--99, 2003.
\newblock \href {https://doi.org/10.1016/S0304-4149(02)00182-5}
  {\path{doi:10.1016/S0304-4149(02)00182-5}}.

\bibitem{GS90}
A.~E. Gelfand and A.~F.~M. Smith.
\newblock Sampling-based approaches to calculating marginal densities.
\newblock {\em J. Amer. Statist. Assoc.}, 85(410):398--409, 1990.
\newblock URL: \url{https://www.jstor.org/stable/2289776}.

\bibitem{JR02}
S.~F. Jarner and G.~O. Roberts.
\newblock Polynomial convergence rates of {M}arkov chains.
\newblock {\em Ann. Appl. Probab.}, 12(1):224--247, 2002.
\newblock \href {https://doi.org/10.1214/aoap/1015961162}
  {\path{doi:10.1214/aoap/1015961162}}.

\bibitem{J16}
D.~C. Jerison.
\newblock {\em The drift and minorization method for reversible Markov chains}.
\newblock PhD thesis, Stanford University, 2016.
\newblock URL:
  \url{http://www.math.tau.ac.il/~jerison/thesis-regeneration.pdf}.

\bibitem{JH01}
G.~L. Jones and J.~P. Hobert.
\newblock Honest exploration of intractable probability distributions via
  {M}arkov chain {M}onte {C}arlo.
\newblock {\em Statist. Sci.}, 16(4):312--334, 2001.
\newblock \href {https://doi.org/10.1214/ss/1015346317}
  {\path{doi:10.1214/ss/1015346317}}.

\bibitem{JH04}
G.~L. Jones and J.~P. Hobert.
\newblock Sufficient burn-in for {G}ibbs samplers for a hierarchical random
  effects model.
\newblock {\em Ann. Statist.}, 32(2):784--817, 2004.
\newblock \href {https://doi.org/10.1214/009053604000000184}
  {\path{doi:10.1214/009053604000000184}}.

\bibitem{LP17}
D.~A. Levin and Y.~Peres.
\newblock {\em Markov chains and mixing times}.
\newblock American Mathematical Society, Providence, RI, 2017.
\newblock Second edition. With contributions by Elizabeth L. Wilmer. With a
  chapter on ``Coupling from the past'' by James G. Propp and David B. Wilson.

\bibitem{LT96}
R.~B. Lund and R.~L. Tweedie.
\newblock Geometric convergence rates for stochastically ordered {M}arkov
  chains.
\newblock {\em Math. Oper. Res.}, 21(1):182--194, 1996.
\newblock \href {https://doi.org/10.1287/moor.21.1.182}
  {\path{doi:10.1287/moor.21.1.182}}.

\bibitem{MH04}
D.~Marchev and J.~P. Hobert.
\newblock Geometric ergodicity of van {D}yk and {M}eng's algorithm for the
  multivariate {S}tudent's {$t$} model.
\newblock {\em J. Amer. Statist. Assoc.}, 99(465):228--238, 2004.
\newblock \href {https://doi.org/10.1198/016214504000000223}
  {\path{doi:10.1198/016214504000000223}}.

\bibitem{MT93}
S.~P. Meyn and R.~L. Tweedie.
\newblock {\em Markov chains and stochastic stability}.
\newblock Communications and Control Engineering Series. Springer-Verlag
  London, Ltd., London, 1993.
\newblock \href {https://doi.org/10.1007/978-1-4471-3267-7}
  {\path{doi:10.1007/978-1-4471-3267-7}}.

\bibitem{MT94}
S.~P. Meyn and R.~L. Tweedie.
\newblock Computable bounds for geometric convergence rates of {M}arkov chains.
\newblock {\em Ann. Appl. Probab.}, 4(4):981--1011, 1994.
\newblock URL: \url{https://projecteuclid.org/euclid.aoap/1177004900}.

\bibitem{M10}
L.~Miclo.
\newblock On absorption times and {D}irichlet eigenvalues.
\newblock {\em ESAIM Probab. Stat.}, 14:117--150, 2010.
\newblock \href {https://doi.org/10.1051/ps:2008037}
  {\path{doi:10.1051/ps:2008037}}.

\bibitem{MTY95}
P.~Mykland, L.~Tierney, and B.~Yu.
\newblock Regeneration in {M}arkov chain samplers.
\newblock {\em J. Amer. Statist. Assoc.}, 90(429):233--241, 1995.
\newblock URL: \url{https://www.jstor.org/stable/2291148}.

\bibitem{N78}
E.~Nummelin.
\newblock A splitting technique for {H}arris recurrent {M}arkov chains.
\newblock {\em Z. Wahrsch. Verw. Gebiete}, 43(4):309--318, 1978.
\newblock \href {https://doi.org/10.1007/BF00534764}
  {\path{doi:10.1007/BF00534764}}.

\bibitem{N84}
E.~Nummelin.
\newblock {\em General irreducible {M}arkov chains and nonnegative operators},
  volume~83 of {\em Cambridge Tracts in Mathematics}.
\newblock Cambridge University Press, Cambridge, 1984.
\newblock \href {https://doi.org/10.1017/CBO9780511526237}
  {\path{doi:10.1017/CBO9780511526237}}.

\bibitem{NT82}
E.~Nummelin and P.~Tuominen.
\newblock Geometric ergodicity of {H}arris recurrent {M}arkov chains with
  applications to renewal theory.
\newblock {\em Stochastic Process. Appl.}, 12(2):187--202, 1982.
\newblock \href {https://doi.org/10.1016/0304-4149(82)90041-2}
  {\path{doi:10.1016/0304-4149(82)90041-2}}.

\bibitem{P97}
I.~Pak.
\newblock {\em Random walks on groups: {S}trong uniform time approach}.
\newblock PhD thesis, Harvard University, 1997.
\newblock URL: \url{https://math.ucla.edu/~pak/papers/time57.pdf}.

\bibitem{RS72}
M.~Reed and B.~Simon.
\newblock {\em Methods of modern mathematical physics. {I}. {F}unctional
  analysis}.
\newblock Academic Press, New York-London, 1972.

\bibitem{R84}
D.~Revuz.
\newblock {\em Markov chains}, volume~11 of {\em North-Holland Mathematical
  Library}.
\newblock North-Holland Publishing Co., Amsterdam, second edition, 1984.

\bibitem{RR97}
G.~O. Roberts and J.~S. Rosenthal.
\newblock Geometric ergodicity and hybrid {M}arkov chains.
\newblock {\em Electron. Comm. Probab.}, 2(2):13--25, 1997.
\newblock \href {https://doi.org/10.1214/ECP.v2-981}
  {\path{doi:10.1214/ECP.v2-981}}.

\bibitem{RR04}
G.~O. Roberts and J.~S. Rosenthal.
\newblock General state space {M}arkov chains and {MCMC} algorithms.
\newblock {\em Probab. Surv.}, 1:20--71, 2004.
\newblock \href {https://doi.org/10.1214/154957804100000024}
  {\path{doi:10.1214/154957804100000024}}.

\bibitem{RT99}
G.~O. Roberts and R.~L. Tweedie.
\newblock Bounds on regeneration times and convergence rates for {M}arkov
  chains.
\newblock {\em Stochastic Process. Appl.}, 80(2):211--229, 1999.
\newblock \href {https://doi.org/10.1016/S0304-4149(98)00085-4}
  {\path{doi:10.1016/S0304-4149(98)00085-4}}.

\bibitem{R95}
J.~S. Rosenthal.
\newblock Minorization conditions and convergence rates for {M}arkov chain
  {M}onte {C}arlo.
\newblock {\em J. Amer. Statist. Assoc.}, 90(430):558--566, 1995.
\newblock URL: \url{https://www.jstor.org/stable/2291067}.

\bibitem{R95b}
J.~S. Rosenthal.
\newblock Rates of convergence for {G}ibbs sampling for variance component
  models.
\newblock {\em Ann. Statist.}, 23(3):740--761, 1995.
\newblock \href {https://doi.org/10.1214/aos/1176324619}
  {\path{doi:10.1214/aos/1176324619}}.

\bibitem{R02}
J.~S. Rosenthal.
\newblock Quantitative convergence rates of {M}arkov chains: a simple account.
\newblock {\em Electron. Comm. Probab.}, 7:123--128, 2002.
\newblock \href {https://doi.org/10.1214/ECP.v7-1054}
  {\path{doi:10.1214/ECP.v7-1054}}.

\bibitem{T94}
L.~Tierney.
\newblock Markov chains for exploring posterior distributions.
\newblock {\em Ann. Statist.}, 22(4):1701--1762, 1994.
\newblock With discussion and a rejoinder by the author.
\newblock \href {https://doi.org/10.1214/aos/1176325750}
  {\path{doi:10.1214/aos/1176325750}}.

\end{thebibliography}
}

\end{document}